\documentclass{amsart}
\usepackage{amssymb}
\usepackage[margin=1.16in]{geometry}
\usepackage[dvipsnames]{xcolor}
\colorlet{GRAY}{gray}
\usepackage{amsmath}
\usepackage{amssymb}
\usepackage{amsthm}
\usepackage[shortlabels]{enumitem}
\usepackage{mathdots}
\usepackage{tikz-cd}
\usepackage{subcaption}

\usepackage{todonotes}
\usepackage{wrapfig}

\usepackage{array}
\usepackage{tabularx}
\usepackage{url} 

\usepackage{pgf,tikz,pgfplots}

\usepackage{longtable}

\usepackage{setspace}

\usetikzlibrary{shapes}
\usepackage{blkarray}
\usepackage{multirow}
\usepackage{mathtools}
\usepackage{graphicx}

\usepackage{hyperref}
\hypersetup{hyperfootnotes=false,bookmarksnumbered,colorlinks=true, allcolors={RoyalBlue}}

\usepackage{theoremref}

\definecolor{light-gray}{gray}{0.95}

\let\amsamp=&

\newcommand{\pcoor}[1]{%
  \begingroup\lccode`~=`: \lowercase{\endgroup
  \edef~}{\mathbin{\mathchar\the\mathcode`:}\nobreak}%
  [
  \begingroup
  \mathcode`:=\string"8000
  #1%
  \endgroup 
  ]
}

\usepackage{float} 

\theoremstyle{definition}
\newtheorem{defn}{Definition}[section]
\newtheorem{theorem}[defn]{Theorem}

\newtheorem{prop}[defn]{Proposition}
\newtheorem{lemma}[defn]{Lemma}
\newtheorem{cor}[defn]{Corollary}
\newtheorem{rk}[defn]{Remark}

\newtheorem*{stepk+1}{Step $\boldsymbol{k+1}$}

\numberwithin{subcase}{case}
 
\newtheorem{notation}[defn]{Notation}

\newtheorem{theoremletter}{Theorem}

\DeclareMathOperator{\Lie}{Lie}
\DeclareMathOperator{\Stab}{Stab}

\def\l{\lambda}

\newcommand{\GG}{{\mathbb G }}
 
\newcommand{\zmin}{Z_{\min}}
\newcommand{\xmino}{X_{\min}^0}
\newcommand{\hX}{\widehat{X}}
\newcommand{\tX}{\widetilde{X}}

\newcommand{\gitq}{/\!/}

\newcommand{\hU}{\widehat{U}}

\newcommand{\GL}{\operatorname{GL}}
\newcommand{\SL}{\operatorname{SL}}
\newcommand{\CC}{\mathbb{C}}

\usetikzlibrary{decorations.pathreplacing,arrows}
\tikzset{
  commutative diagrams/.cd,
  arrow style=tikz
}

\tikzset{
  symbol/.style={
    draw=none,
    every to/.append style={
      edge node={node [sloped, allow upside down, auto=false]{$#1$}}}
  }
}

\tikzset{>=to}

\title[Smoothness of non-reductive fixed point sets and cohomology of quotients]{Smoothness of non-reductive fixed point sets and cohomology of non-reductive GIT quotients }
\author{Eloise Hamilton}

\begin{document}

\begin{abstract}
We establish a method for calculating the Poincar\'e series of moduli spaces constructed as quotients of smooth varieties by suitable non-reductive group actions; examples of such moduli spaces include moduli spaces of unstable vector or Higgs bundles on a smooth projective curve, with a Harder-Narasimhan type of length two. To do so, we first prove a result concerning the smoothness of fixed point sets for suitable non-reductive group actions on smooth varieties. 
This enables us to prove that quotients of smooth varieties by such non-reductive group actions, which can be constructed using Non-Reductive GIT via a sequence of blow-ups, have at worst finite quotient singularities. We conclude the paper by providing explicit formulae for the Poincar\'e series of these non-reductive GIT quotients. \vspace{-0.7cm}
\end{abstract} 

\maketitle

\section{Introduction}

The aim of this paper is to study the topology of moduli spaces which can be constructed as quotients of smooth varieties by suitable non-reductive group actions, such as moduli spaces of unstable vector or Higgs bundles over a smooth projective curve. In order to do this we first prove a result concerning the smoothness of fixed point sets for non-reductive group actions on smooth varieties defined over a field of characteristic zero.

\subsection*{Smoothness of fixed point sets for linear algebraic group actions.} 

If a group $G$ acts on a smooth scheme $X$, an important question for its cohomological implications is whether the fixed point set\footnote{The fixed point set $X^G$ has an induced scheme-theoretic structure, defined in \cite{Fogarty1973}.} $X^G$ is also smooth. By a classical result, this is always true if $G$ is reductive\footnote{Reductivity is sufficient as we are assuming that we are working over a field of characteristic zero. If this is not the case then reductivity must be replaced by linear reductivity.} (see \cite[Prop 1.3]{Iversen1972} or \cite{Fogarty1973}). If $G$ is not reductive, then $X^G$ may not be smooth and indeed explicit counter-examples are given by Fogarty in \cite{Fogarty1973}. In fact, it follows from an equivalent characterisation of reductive groups that the class of such groups is the largest class of linear algebraic groups for which the result can be true: a linear algebraic group $G$ is reductive if and only if for every smooth scheme $X$ equipped with an $G$-action, the fixed point set $X^G$ is smooth \cite{Fogarty1977}\footnote{This result remains valid over a field of arbitrary characteristic provided reductivity is replaced by linear reductivity.}.

The first of the two main results of this paper shows that a positive result concerning smoothness of fixed point schemes can still be obtained for a certain class of non-reductive groups: 
instead of proving that the whole fixed point set is smooth (which cannot be true in general for any class of non-reductive groups), we prove that its intersection with a particular open subset of $X$ is smooth.

\begin{theoremletter}[Smoothness of non-reductive fixed point sets] \thlabel{mainsmoothnessresult}
Let $H = U \rtimes R$ denote a linear algebraic group\footnote{Over a field of characteristic zero any linear algebraic group $H$ can be written as the semi-direct product of its unipotent radical $U$ with a Levi subgroup $R$.} such that its unipotent radical $U$ is abelian and such that $R$ contains a central one-parameter subgroup $\lambda: \GG_m \to Z(R)$ acting with a single and positive weight on the Lie algebra of $U$ via the adjoint action. Suppose that $H$ acts on a projective scheme $X$, and let $\xmino$ denote the open Bialynicki-Birula stratum associated to the action of $\lambda(\GG_m)$ on $X$. Then $X^H \cap \xmino$ is smooth at any point at which $\xmino$ is smooth. In particular, the scheme $X^H \cap \xmino$ is smooth if $X$ is smooth. 
\end{theoremletter}

The groups considered in the above \thref{mainsmoothnessresult} are examples of `internally graded' groups; these are linear algebraic groups containing a central one-parameter subgroup acting with positive weights on the Lie algebra of their unipotent radical via the adjoint action. The relevance of internally graded groups is that they are the groups to which Non-Reductive Geometric Invariant Theory (GIT) applies. That is, recent results in Non-Reductive GIT established in \cite{Berczi2020} show that classical GIT has an effective analogue for internally graded groups, enabling the construction of quotients for actions of such groups. Our proof of \thref{mainsmoothnessresult} will use results from Non-Reductive GIT.

\subsection*{Group actions and cohomology.} Fixed point sets play an important role in understanding the cohomology of a space $X$ equipped with the action of a group $G$. Indeed, the cohomology of $X$ (either ordinary or equivariant) can often be related to that of the fixed point set $X^G$, which is typically easier to describe. A significant result in this direction is the Bialynicki-Birula decomposition for the action of an algebraic torus $T$ on a smooth projective variety $X$, which leads to a formula for the integral homology of $X$ in terms of the homology of the fixed point set $X^T$ \cite{Bialynicki-Birula1973,BialynickiBirula1974,Carrell1979}\footnote{Various generalisations of the results of \cite{Bialynicki-Birula1973,BialynickiBirula1974,Carrell1979} exist. For example, they can be extended to the case where $X$ is singular under certain additional assumptions (see \cite{Carrell1983,Kirwan1988}; note that in \cite{Kirwan1988} intersection cohomology is considered instead of singular cohomology). Moreover, the Bialynicki-Birula decomposition has been extended to the actions of linearly reductive groups on schemes of finite type and on algebraic spaces -- see \cite{Jelisiejew2019}.}. Other important results include the Atiyah-Bott and Beline-Vergne localisation theorems for the actions of compact Lie groups on smooth compact manifolds (see \cite{Berline1982,Atiyah1984}), which relate the equivariant cohomology ring of the manifold to that of the fixed point set (see also \cite{Edidin1998}). These localisation theorems are particularly useful as they lead to explicit formulae for integrals over equivariant cohomological classes in terms of integrals over the fixed point set, which are simpler to compute. 
Such formulae have indeed been applied to great effect in enumerative algebraic geometry \cite{Kontsevich1994,Ellingsrud1994}, symplectic geometry \cite[\S VI]{Audin2004} and mathematical physics \cite{Cordes1995,Szabo1996,Ferro2018} for example. 

The results mentioned above all rely on specific assumptions on the group, requiring reductivity of the group at a minimum; this is closely related to the fact that if $G$ is reductive then $X^G$ is smooth. A positive result such as \thref{mainsmoothnessresult} in the non-reductive case is useful therefore for studying the cohomology of spaces equipped with the action of non-reductive groups, and the second main result of our paper is an example of this. Indeed, \thref{mainsmoothnessresult} is a stepping stone to proving a general result concerning the cohomology of non-reductive GIT quotients, which is the main aim of this paper.

\subsection*{Poincar\'e series of non-reductive GIT quotients.}
Given the linear action of an internally graded linear algebraic group $H$ on an projective scheme $X$, Non-Reductive GIT enables the construction of a quotient for the action of $H$ on an open subset of $X$ with an explicit projective completion. In general\footnote{Blow-ups are required in Non-Reductive GIT when `semistability does not coincide with stability', a condition analogous to the corresponding condition in classical GIT.} this construction involves performing a sequence of blow-ups of $X$, analogous to the partial desingularisation construction of classical GIT \cite{Kirwan1985}. 
The procedure results in a projective scheme $\hX$ which admits a projective geometric quotient $X \gitq \hU$ for the action of $\hU: = U \rtimes \lambda(\GG_m)$ on an explicitly determined open subset of $\hX$, where $U$ is the unipotent radical of $H$ and $\lambda$ denotes the grading one-parameter subgroup. A projective good quotient $\hX \gitq H$ of an open subset of $\hX$ by the action of $H$ can then be obtained by considering the classical GIT quotient for the action of the reductive quotient group $H/\hU$ on the projective scheme $\hX \gitq \hU$; by construction this quotient is a projective completion of a geometric quotient for the action of $H$ on an open subset of $X$.

The second main result of this paper is to provide a formula for the Poincar\'e series of $\hX \gitq \hU$, and of $\hX \gitq H$ under additional assumptions, when $X$ is a smooth complex projective variety and $H$ is of the form given in \thref{mainsmoothnessresult}. \thref{mainsmoothnessresult} is essential to obtaining this formula as it is used to prove that the centres of the non-reductive blow-ups at each stage are smooth. The results we obtain can be summarised as follows (for a precise formulation see \thref{nonredformula} and \thref{mostgeneralformula}):

\begin{theoremletter}[Poincar\'e series of non-reductive GIT quotients]  \thlabel{maintheorem}
Let $H$, $X$ and $\xmino$ be as per \thref{mainsmoothnessresult} and defined over the complex numbers. Let $\zmin \subseteq X$ denote the closed subvariety of $X$ corresponding to the image of $\xmino$ under the map $x \mapsto \operatorname{lim}_{t \to 0} t \cdot x$ for $t \in \lambda(\GG_m)$ and $x \in X$. Suppose in addition that $X$ is smooth, that $H$ acts linearly on $X$ and that there exists a point in $\zmin$ with trivial stabiliser group in $U$\footnote{The latter assumption can be viewed as the non-reductive analogue of the assumption in the partial desingularisation construction that the stable locus is non-empty.}. 
Let $\hX$ denote the scheme resulting from the sequence of non-reductive blow-ups from the action of $\hU : = U \rtimes \lambda(\GG_m)$ on $X$. Then: \begin{enumerate}[(i)]\item \label{part1} the projective geometric quotient $\hX \gitq \hU$ has at worst finite quotient singularities and its Poincar\'e series can be expressed in terms of those of $\zmin$ and of iterated blow-ups of $\zmin$;
\item \label{part2} the intersection of the centre of the blow-up with the corresponding $\zmin$ at each stage is smooth, and can be identified as a resolution of singularities of an explicit closed subvariety of $\zmin$;
\item \label{part3} if the semistable and stable loci coincide for the induced action of $R_{\l} : = R/\lambda(\GG_m)$ on $\zmin$, then the projective geometric quotient $\hX \gitq H$ has at worst finite quotient singularities and its Poincar\'e series can be expressed in terms of those of $\zmin \gitq R_{\l}$ and of iterated blow-ups of $\zmin \gitq R_{\l}$. 
\end{enumerate}
\end{theoremletter}

\subsection*{Applications to unstable bundles and Brill-Noether theory.} 

\thref{maintheorem} illustrates the general principle that viewing a moduli space as a quotient of a parameter space by a group action, which is in fact how most moduli spaces are constructed, is a powerful perspective for studying its cohomology. For moduli spaces which can be constructed as classical GIT quotients, this approach is pursued in \cite{Kirwan1985,Kirwan1986} and indeed \thref{maintheorem} can be viewed as  non-reductive analogue of the formula obtained in \cite{Kirwan1985}. The results of \cite{Kirwan1985,Kirwan1986} have been used for example to compute the Poincar\'e series (for ordinary or intersection cohomology) of a number of moduli spaces in algebraic geometry, including moduli spaces of products of Grassmannians \cite[\S 16]{Kirwan1984}, of vector bundles on a smooth projective curve \cite{Kirwan1986}, of K3 surfaces \cite{Kirwan1989}, of hypersurfaces in projective spaces \cite{Kirwan1989a,Casalaina-Martin2019}, and more recently of certain genus $4$ curves \cite{Fortuna2020} and of pure sheaf spaces \cite{Chung2021}. 

\thref{maintheorem} can be used to calculate the Poincar\'e series of moduli spaces which can be constructed as non-reductive GIT quotients of a smooth complex projective variety $X$ by the action of a group $H$ of the form given in \thref{mainsmoothnessresult}. 
This is the case for moduli spaces of unstable vector or Higgs bundles on a smooth projective curve (and more generally of unstable sheaves or Higgs sheaves on a smooth projective variety) with a fixed coprime Harder-Narasimhan type of length two\footnote{A Harder-Narasimhan type $\mu=(d_1/r_1,\hdots,d_1/r_1,d_2/r_2,\hdots,d_s/r_s)$ is coprime if $d_i$ and $r_i$ are coprime for each $i$. The length of $\mu$ corresponds to the integer $s$.}, as constructed in \cite{Brambila-Paz2009,Jackson2018,Hamilton2019a}.
In this setting, all of the conditions of \thref{mainsmoothnessresult} are satisfied and thus the Poincar\'e series of these moduli spaces can be computed from the Poincar\'e series of $\zmin \gitq R_{\l}$ and of iterated blow-ups of it using \ref{part3} of \thref{maintheorem}\footnote{To be more precise, \thref{maintheorem} enables the calculation of the Poincar\'e series of a partial compactification of the moduli spaces. The partial compactification is of the form $\hX \gitq H$ and by construction contains as an open subset a geometric quotient for the action of $H$ on an open subset of the parameter space $X$; the moduli space corresponds to this geometric quotient.}.

In this case the variety $\zmin \gitq R_{\l}$ corresponds to the moduli space of unstable vector or Higgs bundles which are isomorphic to their Harder-Narasimhan graded. In other words it is the product of two moduli spaces of semistable bundles of lower rank, the cohomology of which has and continues to be widely studied. Moreover, the centres of the blow-ups of $\zmin \gitq R_{\l}$ at each stage can be interpreted thanks to part \ref{part2} of \thref{maintheorem} as a partial resolution of singularities\footnote{The resolution of singularities is only partial because of the possible presence of finite quotient singularities arising from taking the quotient by $R_{\l}$.} of certain Brill-Noether loci associated to the base curve. If the bundles have rank two, then the corresponding  Brill-Noether loci are of rank one, the theory of which is well-known \cite{Arbarello2013}. This makes it feasible to obtain an explicit formula for the Poincar\'e series of moduli spaces of unstable Higgs bundles of rank two, a first step towards determining whether the cohomology of these moduli spaces is as rich as that of the moduli space of semistable Higgs bundles, which represents an active area of research (see for example \cite{Hitchin1987,Hausel2003a,Rayan2018,Maulik2021}).  
For higher rank bundles, the centres of the blow-ups at each stage represent partial desingularisations of higher rank Brill-Noether loci, which are far from fully understood (see for example \cite{Bigas1991,Bradlow2003,Osserman2013}); \thref{maintheorem} may therefore help shed new light on these loci. We will address the application of \thref{maintheorem} to unstable vector and Higgs bundles and its link with Brill-Noether theory in a separate paper.

\subsection*{Cohomology ring of quotients.} 
While we focus in this paper on the particular cohomological invariant given by the Poincar\'e series, this is not the only cohomological information which can be extracted from considering a moduli space as a quotient. In classical GIT, when semistability coincides with stability, the surjectivity of the Kirwan map allows generators of the cohomology ring of the GIT quotient to be obtained from generators of the equivariant cohomology ring of the semistable locus \cite{Kirwan1984,Kirwan1992}. Moreover, the problem of determining intersection pairings in the cohomology ring can be simplified by reducing to a maximal torus in the reductive group, thanks to non-abelian localisation theorems (see \cite{Jeffrey1995,Kalkman1996,Martin1999}). These results can be generalised to the case where semistability does not coincide with stability by replacing ordinary cohomology by intersection cohomology -- see \cite{Jeffrey2001}. An important application of these results is the computation of the cohomology ring of the moduli space of vector bundles on a smooth projective curve (see \cite{Jeffrey1998} for the coprime case and \cite{Jeffrey2005} for the non-coprime case). 

In Non-Reductive GIT, when a condition analogous to the condition that semistability coincides with stability is satisfied, it is shown in \cite{Berczi2019} that methods  similar to the classical case can be used to compute the cohomology ring of non-reductive GIT quotients. The results of \cite{Berczi2019} have been used in \cite{Berczi2019a} to prove the polynomial Green-Griffiths-Lang and Kobayashi conjectures. Our hope is that the results of \cite{Berczi2019} regarding the cohomology ring of non-reductive GIT quotients can be extended to the general non-reductive case, namely when blow-ups are required to construct the quotient, to shed light on the structure of the cohomology ring of such quotients beyond their Poincar\'e series.

The method of abelian localisation for computing the cohomology ring of a quotient has been generalised in \cite{Hausel2005} to the case of hyperk\"ahler quotients. Such quotients can arise in classical GIT when considering the induced action on the cotangent bundle of the parameter space; examples include hypertoric varieties, quiver varieties and hyperpolygon spaces. Results from \cite{Proudfoot2011} show that information about the cohomology ring of the quotient of the cotangent bundle can be extracted from that of the initial quotient. We hope in future work to investigate whether non-reductive counterparts to these results exist.

\subsection*{Structure of the paper.} In Section \ref{sec:reviewofnrgit} we summarise the results of Non-Reductive GIT which we use in this paper. 
In Section \ref{sec:smoothness} we prove  \thref{mainsmoothnessresult} using results from Non-Reductive GIT. 
In Section \ref{sec:cohomss=s} we summarise existing results concerning the computation of the Poincar\'e series of GIT quotients in the case where `semistability coincides with stability' (in either the classical or the non-reductive sense). 
In Section \ref{sec:cohomssneqspart2} we generalise these results to the case where `semistability does not coincide with stability' by proving \thref{maintheorem}.

\subsection*{Conventions.} 

In this paper we work over a field $k$ of characteristic zero, specialising to the case where $k= \mathbb{C}$ from Section \ref{sec:cohomss=s} onwards. By a scheme we mean a scheme of finite type over $k$. For cohomological purposes we will work mostly with smooth schemes, and since smooth (connected) schemes are reduced and irreducible, for simplicity we will work with varieties throughout. We do not assume that all varieties are irreducible and will add the qualifier when needed.

\subsection*{Acknowledgements.} Most of the work presented in this paper was completed during my DPhil under the supervision of Professor Frances Kirwan and I am profoundly grateful for all her support and guidance. I would also like to thank Gergely B\'erczi and David Rydh for many helpful conversations.

\section{Review of Non-Reductive GIT}  \label{sec:reviewofnrgit}

In this section we summarise the main results of Non-Reductive Geometric Invariant Theory (GIT). Section \ref{subsec:motivandsetup} introduces the set-up and notation required to formulate the results of Non-Reductive GIT. Section \ref{subsec:ss=snonred} describes the main results of Non-Reductive GIT under the assumption that `semistability coincides with stability', while Section \ref{subsec:ssneqsnonred} considers the case where this condition is not satisfied.

\subsection{Set-up for Non-Reductive GIT} \label{subsec:motivandsetup}

Given the linear action of a reductive group $G$ on a projective variety $X$, three key features of classical GIT are: \begin{enumerate}[(i)]
 \item  the existence of a projective GIT quotient $X \gitq G$ obtained as the projective spectrum of the ring of (finitely generated) invariants; 
 \item  the good quotient map from the semistable locus $X^{ss}$ to the GIT quotient $X \gitq G$ induced by the inclusion of the invariants and which restricts to a geometric quotient on the stable locus $X^s$; 
 \item the Hilbert-Mumford criterion which allows the computation of the semistable locus without having to find invariants. 
 \end{enumerate}

All three features rely on the reductivity of the group, and if $G$ is no longer reductive each feature can indeed fail. Explicit examples of each of these failures are given in \cite{Doran2007}, which represents the starting point of Non-Reductive GIT. 

\subsubsection*{Towards Non-Reductive GIT} \cite{Doran2007} addresses the problem of generalising existing methods from GIT to linear group actions by non-reductive groups. Notions of semistability and stability for the action of a linear algebraic group $H$ on a projective\footnote{Semistability and stability are defined in \cite{Doran2007} for arbitrary varieties, not necessarily projective. Nevertheless, the projectivity assumption is necessary for many of the results obtained in \cite{Doran2007}.} variety $X$ are defined (notions which reduce to the classical notions when the group is reductive), and the existence of a geometric quotient for action of $H$ on the stable locus and of a canonical `enveloping quotient' $X \gitq H$ of the semistable locus is proved. 
Moreover, it is shown how the semistable and stable loci can be computed explicitly from the Hilbert-Mumford criterion applied to the reductive group action used to define the projective completion of the enveloping quotient $X \gitq H$. 

Nevertheless, the enveloping quotient may not be projective (if the invariants are not finitely generated) and moreover the map from the semistable locus to the enveloping quotient may not be surjective. The introduction of a `grading' multiplicative group, first considered in \cite{Berczi2016a}, solves both issues simultaneously and by doing so enables a generalisation of GIT to non-reductive group actions which essentially preserves all of the features of classical GIT. It is this generalisation which gives rise to Non-Reductive GIT, the main results of which appear in \cite{Berczi2020}.

\subsubsection*{Role of the grading $\GG_m$} Instead of considering the action of a general linear algebraic group $H$, Non-Reductive GIT considers the action of a semi-direct product $\widehat{H}: =  H \rtimes \GG_m$ where the multiplicative group $\GG_m$ acts with strictly positive weights on the Lie algebra of the unipotent radical $U$ of $H$ (called a \emph{grading} $\GG_m$). 

Given that a linear algebraic group\footnote{We recall that we are working over a field of characteristic zero.} $H$ can be written as a semidirect product $U \rtimes R$ of its unipotent radical with a reductive subgroup, Non-Reductive GIT constructs a quotient for the linear action of $\widehat{H}$ on a projective variety $X$ `in stages': first by constructing a quotient for the action of $\widehat{U} : = U \rtimes \GG_m$ on $X$, and then by using classical GIT to construct a quotient of the resulting projective quotient by an induced action of the reductive group $R$. Thus the crux of the theory consists in constructing quotients by groups of the form $\hU = U \rtimes \GG_m$ where $U$ is a unipotent group and the multiplicative group $\GG_m$ acts on $\Lie U$ via the adjoint action with positive weights. We call such groups \emph{externally graded unipotent groups}.

The grading $\GG_m$ is used in two fundamental ways to construct a projective geometric quotient $X \gitq \hU$ for the action of $\hU$ on an explicitly determined open subset of $X$ (assuming that a certain condition regarding unipotent stabiliser groups is met). Firstly, it is used to define an open subset of $X$ which admits a locally trivial $U$-quotient. 
Secondly, it is used to define a linear $\GG_m$-action on a suitable projective completion of this $U$-quotient. 
By construction, the GIT quotient for the linear $\GG_m$-action on the projective completion of the $U$-quotient is a projective variety admitting a surjective map from an open subset of the initial variety $X$, and this open subset can be explicitly determined thanks to the Hilbert-Mumford criterion applied to the action of $\GG_m$ on the projective completion of the $U$-quotient. In fact, the linearisation of the $\GG_m$-action on the projective completion of the $U$-quotient is constructed in such a way that the resulting projective GIT quotient is a geometric quotient for the action of $\hU$ on the explicitly determined open subset of $X$.

\subsubsection*{Key definitions.}
Formulating the results of Non-Reductive GIT requires introducing the following definitions. Let $\hU := U \rtimes \GG_m$ denote an externally graded unipotent group and suppose that $\hU$ acts linearly on a projective variety $X$ with ample line bundle $L$. By taking a tensor power of $L$ if necessary we can assume that it is very ample, so that $X \subseteq \mathbb{P}(V)$ where $V = H^0(X,L)^{\vee}$. Let $\omega_{\operatorname{min}} = \omega_0 < \omega_1 < \cdots < \omega_{\operatorname{max}}$ denote the weights with which $\lambda(\mathbb{G}_m)$ acts on $V$ and let $V_{\operatorname{min}}$ denote the minimal weight space for the action of $\lambda(\mathbb{G}_m)$ on $V$. We then define a closed subvariety $\zmin$ of $X$ by $$Z_{\operatorname{min}} := X \cap \mathbb{P}(V_{\operatorname{min}})$$ and an open subvariety $\xmino$ of $X$ by $$X^0_{\operatorname{min}} : = \left\{ x \in X \ | \ \operatorname{lim}_{t \to 0} t \cdot x \in Z_{\operatorname{min}} \right\}.$$ Note that there is a natural retraction map $p: \xmino \to \zmin$ given by $x \mapsto \operatorname{lim}_{t \to 0} \lambda(t)\cdot x$ for $t \in \GG_m$.

\begin{notation}[Notation for $\zmin$ and $\xmino$]
If $\hU$ acts on another projective variety $Y$ (i.e.\ a variety not denoted by $X$), then we let $Y^0_{\operatorname{min}}$ and $Z(Y)_{\operatorname{min}}$ denote the analogues for $Y$ of $\xmino$ and $\zmin$ respectively.  
\end{notation}

The linearisation of the action of $\hU$ on $X$ is \emph{adapted} if $\omega_{\operatorname{min}} < 0 < \omega_1$. We note that by taking a positive tensor power and twisting the linearisation by an appropriate character, we can always assume that the linearisation of the $\hU$-action on $X$ is adapted.

\subsection{When `semistability coincides with stability'} \label{subsec:ss=snonred}

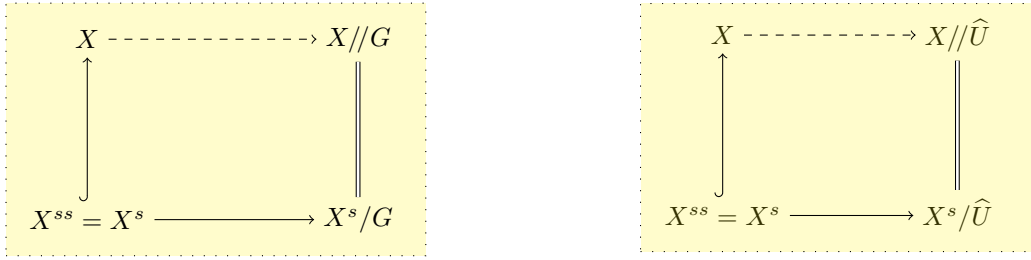
\begin{figure}
\begin{subfigure}{0.45\textwidth}
\centering
        \begin{tikzpicture}[scale=0.8,x=(15:8cm),y=(0:4.5cm), z=(90:1.5cm),
    cross line/.style={preaction={draw=white,-,line width=3pt}},
    equal/.style={double distance=1pt}]

\node (c1) at (1,-1.30,-0.4) {  } ;
\node (c2) at (1,-1.30,2.4) { };
\node (c3) at (1,0.25,-0.4) { };
\node (c4) at (1,0.25,2.4) { };

\filldraw  [fill = yellow, draw=black,loosely dotted, opacity=0.2] (c1) rectangle (c4);

\node (16) at (1,0,2) {\(   X \gitq G  \)};
\node(19) at (1,-1,2) {\( X \)};
\node(19') at (1,0,0) {\( X^s / G \)};
\node(20) at (1,-1,0) {\(  X^{ss} = X^{s}    \)};

\draw [black,loosely dotted] (c1) rectangle (c4);
\draw [dashed, ->] (19) -- (16);
\draw[ ->] (20) -- (19');
\draw [equal] (19') -- (16);
\draw [right hook ->] (20) -- (19);

		\end{tikzpicture}
        \caption{GIT for the linear action of a reductive group $G$ on a projective variety $X$, when the equality $X^{ss} = X^{s}$ is satisfied.} \label{fig:classicalgit1} 
	\end{subfigure} \hspace{1cm}
\begin{subfigure}{0.45\textwidth}
       \centering
       \begin{tikzpicture}[scale=0.8,x=(15:8cm),y=(0:3.9cm), z=(90:1.5cm),
    cross line/.style={preaction={draw=white,-,line width=3pt}},
    equal/.style={double distance=1pt}]
  
\node (c1) at (1,-1.35,-0.4) {  } ;
\node (c2) at (1,-1.35,2.35) { };
\node (c3) at (1,0.35,-0.4) { };
\node (c4) at (1,0.35,2.35) { };

\node (16) at (1,0,2) {\(   X \gitq \hU \)};
\node(19) at (1,-1,2) {\( X \)};
\node(20) at (1,-1,0) {\(   X^{ss} = X^{s} 
 \)};
\node (19') at (1,0,0) {\( X^s / \hU \)};
     
\filldraw  [fill = yellow, draw=black,loosely dotted, opacity=0.2] (c1) rectangle (c4);
  
\draw [black,loosely dotted] (c1) rectangle (c4);
\draw [equal] (16) -- (19');
\draw [dashed, ->] (19) -- (16);
\draw[right hook ->] (20) -- (19);
\draw [->] (20) -- (19');
     
		\end{tikzpicture}
        \caption{GIT for the linear action of an externally graded unipotent group $\hU$ on a projective variety $X$, when \eqref{ss=s(U)} is satisfied.} \label{fig:unip1}
    \end{subfigure} 

  \caption{Comparison of GIT for reductive and for  externally graded unipotent groups when `semistability coincides with stability'.} \label{fig:comparison}
\end{figure}

The building block of Non-Reductive GIT is the $\hU$-theorem (see \cite[Thm 2.16]{Berczi2020}). The theorem states that if the linear action of $\widehat{U}$ on $X$ satisfies an additional condition (analogous to the condition that semistability coincides with stability in classical GIT), then after taking a positive tensor power and twisting the linearisation by a suitable rational character, all of the properties of classical GIT in the case when semistability coincides with stability can be recovered. The precise formulation is as follows.

\begin{theorem}[$\hU$-theorem] \thlabel{uhatthm} 
Let $\hU = U \rtimes \mathbb{G}_m$ where $\GG_m$ acts with strictly positive weights on $\operatorname{Lie} U$ via the adjoint action and suppose that $\hU$ acts linearly on an irreducible projective variety $X$ in such a way that the linearisation is adapted. 
Then, if the condition
\begin{equation} 
\label{ss=s(U)}
\Stab_U (z)= \{e\} \text{ for all $z \in \zmin$} \tag{$ss=s \neq \emptyset [\hU]$}
\end{equation} is satisfied\footnote{This condition is analogous to the condition in classical GIT that $X^{ss} = X^{s} \neq \emptyset$, which is why it is denoted \eqref{ss=s(U)}.}, we have:
\begin{enumerate}[(i)]
\item there is exists a projective geometric quotient $$ X^s := \xmino \setminus U \zmin \to X \gitq \hU$$ for the action of $\hU$ on the open subset $\xmino \setminus U \zmin$ of $X$, so that set-theoretically $X \gitq \hU = X^s / \hU$;
\item there exists an $\epsilon >0$ such that if linearisation is modified\footnote{This can be achieved by taking a tensor power and twisting by a suitable character.} so that the inequality $\omega_{\operatorname{min}} < 0  < \omega_{\operatorname{min}} + \epsilon   < \omega_1$ is satisfied, then the resulting algebra of invariants $\bigoplus_{k \geq 0} H^0(X,L^{\otimes k})^{\hU}$ is finitely generated and its associated projective variety is isomorphic to $X \gitq \hU$. \label{fginvariants}
\end{enumerate}
\end{theorem} 

It follows from \thref{uhatthm} that the projective variety $X \gitq \hU$ satisfies all of the key properties of a classical GIT quotient in the case where semistability coincides with stability, as illustrated by Figure \ref{fig:comparison}. It is important to note however that by contrast with classical GIT where the (semi)stable locus $X^{(s)s}$ and projective variety $X \gitq G$ are always well-defined, in Non-Reductive GIT the additional condition that \eqref{ss=s(U)}, which should be viewed as the analogue of the condition in classical GIT that the semistable and stable loci coincide, is required to ensure that the stable locus $X^s$ and projective variety $X \gitq \hU$ are well-defined.

\subsubsection*{Extending the $\hU$-theorem to groups with internally graded unipotent radical.} The $\hU$-theorem can be combined with classical GIT to construct quotients for linear actions of \emph{internally graded} linear algebraic groups. These are linear algebraic groups $H = U \rtimes R$ (here $U$ denotes the unipotent radical) containing a grading $\GG_m$, that is, a one-parameter subgroup $\lambda:\GG_m \to Z(R)$ where $Z(R)$ denotes the centre of $R$ and such that the adjoint action of $\lambda(\GG_m)$ on $\operatorname{Lie} U$ has positive weights. 

Let $H = U \rtimes R$ be an internally graded linear algebraic group, and let $\lambda: \mathbb{G}_m \to Z(R)$ denote the grading one-parameter subgroup.  Suppose that $H$ acts linearly on an irreducible projective variety $X$ and that \eqref{ss=s(U)} is satisfied for the action of $\hU$ on $X$. Then a projective quotient for the action of $H$ on (an open subset of) $X$ can be constructed by quotienting in stages: first by the action of $\hU : = U \rtimes \lambda(\GG_m)$ on $X$, then by the action of $R_{\l}:= R /\lambda(\GG_m)$ on $X \gitq \hU$. The key for the second step is the fact that the projective variety $X \gitq \hU$ has an induced action of the reductive group $R_{\l}$ which can be linearised in such a way that the pull-back of this linearisation to $X$ under the quotient map coincides with a tensor power of the linearisation for the $\hU$-action on $X$ (after modifying the original linearisation according to the $\hU$-theorem). Thus we obtain a projective variety $X \gitq H$ given by $( X \gitq \hU) \gitq  R_{\l}$. By defining 
\begin{equation} \label{nrss} X^{(s)s}  := q_{\hU}^{-1}( (X \gitq \hU)^{(s)s}
\end{equation} where $q_{\hU}: \xmino \setminus U \zmin \to X \gitq \hU$ denotes the quotient map and $(X \gitq \hU)^{(s)s}$ the (semi)stable locus for the induced action of $R_{\l}$ on $X \gitq \hU$ , we have that $X \gitq H$ is a good quotient for the action of $H$ on $X^{ss}$ and that $(X \gitq \hU)^{s}/R_{\l} =  X^{ss} / H$ is a geometric quotient for the action of $H$ on $X^{s}$. Figure \ref{fig:gitingeneral}, which combines Figures \ref{fig:classicalgit1} and \ref{fig:unip1}, illustrates these results.

\begin{figure}
\centering
\begin{tikzpicture}[scale=0.9, x=(10:10cm),y=(0:5.2cm), z=(90:1.3cm),
    cross line/.style={preaction={draw=white,-,line width=3pt}},
    equal/.style={double distance=1pt}]

\node (c9) at (0,-1.35,2.5) {};
\node (c11) at (0,1.55,2.5) {};
\node (c10) at (0,-1.35,-1.5) {};
\node (c12) at (0,1.55,-1.5) { };
      
\node (8) at (0,-1,2) {\( X \)};
\node (9) at (0,0,2) {\( X \gitq \hU \)};
\node (10) at (0,1,2) {\( \left( X \gitq \hU \right)   \gitq R_{\l} =: X \gitq H \)};
\node (11) at (0,-1,1) {\(  X^{s} \)};
\node(11') at (0,-1,0) {\( X^{ss}\)};
\node (12') at (0,-1,-1) {\( X^{s} \)};
\node (12) at (0,0,1) {\( X^{s} / \hU \)};
\node (13) at (0,0,0) {\( \left(X \gitq \hU \right)^{ss} \)};
\node (14) at (0,0,-1) {\(\left(X \gitq \hU \right)^{s} \)};
\node (15) at (0,1,-1) {\( \left( X \gitq \hU \right)^{s} / R_{\l}  = X^{s} /H \)};

\draw  [black,loosely dotted] (c10) rectangle (c11);
\filldraw  [fill = orange, draw=black,loosely dotted, opacity=0.2] (c10) rectangle (c11);

\draw [right hook ->](11) -- (8);
\draw [equal](12) -- (9);
\draw [right hook ->] (12') -- (11');
\draw [right hook ->] (11') -- (11);
\draw [dashed,->] (8) -- (9);
\draw [dashed, ->] (9) -- (10);
\draw [right hook ->] (15) -- (10);
\draw [right hook ->] (13) -- (12);
\draw [right hook ->] (14) -- (13);
\draw [->] (14) -- (15);
\draw [->] (12') -- (14);
\draw [->] (13) -- (10);
\draw [->] (11') -- (13);

\draw[->](11) -- (12);

\end{tikzpicture}
\caption{GIT for linear algebraic groups $H = U \rtimes R$ with internally graded unipotent radical, when \eqref{ss=s(U)} is satisfied. The grading one-parameter subgroup is denoted by $\lambda:\GG_m \to Z(R)$, and $R_{\l}$ denotes the quotient $R / \lambda(\GG_m)$.}  
\label{fig:gitingeneral}
\end{figure}
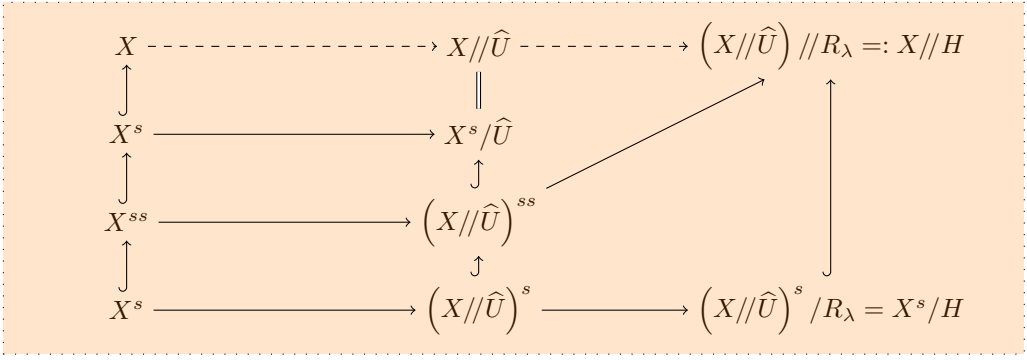

\subsubsection*{Hilbert-Mumford-type criterion for Non-Reductive GIT} As mentioned in Section \ref{subsec:motivandsetup}, an important feature of classical GIT is the Hilbert-Mumford criterion which provides a way of describing the semistable locus without having to compute invariants. That is, if a reductive group $G$ acts linearly on a projective variety $X$, then there is an equality $$X^{ss} = \bigcap_{g \in G} g X^{ss,T}$$ where $X^{ss,T}$ denotes the semistable locus for the restricted linear action of a maximal torus $T \subseteq G$. The advantage of this description is that semistable and stable loci for the action of tori can be computed in a combinatorial way. 

It is shown in \cite{Berczi2020} that an analogue of the Hilbert-Mumford criterion remains valid in Non-Reductive GIT (under the assumption that \eqref{ss=s(U)}). That is, \cite[Thm 2.16]{Berczi2020} establishes the analogous equality \begin{equation} X^{(s)} = \bigcap_{h \in H} h X^{(s)s,T} \label{hmcriterionH}
\end{equation}  for a fixed choice of maximal torus $T \subseteq R$.

\subsection{When `semistability does not coincide with stability'} \label{subsec:ssneqsnonred}

The $\hU$-theorem requires the assumption that $\eqref{ss=s(U)}$ is satisfied for the action of $\hU$ on $X$. 
If this condition is not satisfied, then at present it is not known whether there exists a projective GIT quotient that is a good quotient for the action of $\hU$ on an open subset of $X$ satisfying a Hilber-Mumford-type explicit description. 
Nevertheless, if \eqref{ss=s(U)} is not satisfied, then a construction analogous to the partial desingularisation construction of classical GIT can be applied to obtain a new variety with a linear $\hU$-action such that \eqref{ss=s(U)} is satisfied.

 In classical GIT, given the action of a reductive group $G$ on a projective variety $X$ such that the stable locus (assumed to be non-empty) is strictly contained in the semistable locus, the partial desingularisation construction of \cite{Kirwan1985} can be applied to produce a variety $\tX$ with a linear $G$-action such that semistability coincides with stability. The construction consists in a sequence of blow-ups of $X$ along loci of points with maximal dimension reductive stabiliser groups, which after a finite number of steps results in a variety $\tX$ with no semistable points fixed by a positive-dimensional reductive subgroup of $G$. This suffices to ensure that $\tX^{ss} = \tX^s$.

In Non-Reductive GIT, given the action of an externally graded unipotent group $\hU$ on an irreducible projective variety $X$, if \eqref{ss=s(U)} is not satisfied then a construction analogous to the partial desingularisation can be applied. 
The construction results in a variety $\hX$ with a linear $\hU$-action such that \eqref{ss=s(U)} is satisfied, as per \thref{uhatthm} below.

\begin{theorem}[$\hU$-theorem with blow-ups] \thlabel{uhatwbups}
Let $\hU$ and $X$ be as per \thref{uhatthm} above. 
If the condition \eqref{ss=s(U)} is not satisfied, but the condition that \begin{equation} \label{snonempty}
\text{there exists $z \in \zmin$ such that $ \operatorname{Stab}_U (x ) = \{e\}$} \tag{$\emptyset \neq s \subsetneq ss [\hU]$}
\end{equation} is satisfied, then there exists a sequence of blow-ups of $X$ along $\hU$-invariant closed subvarieties resulting in a projective variety $\hX$ with a linear action of $\hU$ for which the condition \eqref{ss=s(U)} is satisfied. Moreover, the blow-down map is an isomorphism over the subvariety $$X^{\widehat{s}} : = \left\{ x \in \xmino \setminus U \zmin \ \left| \ \operatorname{Stab}_U (x) = \{e\}\right\}\right.,$$ which admits a geometric $\hU$-quotient $X^{\widehat{s}} /U$, open inside $\hX \gitq \hU$.   
\end{theorem}

The `hat' superscript in the notation $X^{\widehat{s}}$ indicates that the quasi-projective geometric quotient of $X^{\widehat{s}}$ by the action of $\hU$ has an explicit projective completion constructed from a blow-up of $\hX$ of $X$. Figure \ref{fig:comparisonssneqsuhat} illustrates the above \thref{uhatwbups}, while Figure \ref{fig:comparisonssneqsred} illustrates the partial desingularisation construction of classical GIT, to allow a visual comparison of the two. As noted above, the only notable difference is that in classical GIT, even when semistability does not coincide with stability, a projective GIT quotient exists (the quotient $X \gitq G$ appearing in the top right-hand corner of the orange rectangle), whereas in Non-Reductive GIT, if the condition \eqref{ss=s(U)} is not satisfied, then it is not known at present whether an analogous projective GIT quotient always exists (as indicated by the absence of a quotient in the top right-hand corner of the red rectangle).

As we will need an explicit description of the centres of the blow-ups referred to in \thref{uhatwbups} to show that they are smooth in Section \ref{sec:cohomssneqspart2}, we describe the construction below. 

\subsubsection*{The blow-up construction for externally graded unipotent groups.} Let $\hU$ be an externally graded unipotent group acting linearly on an irreducible projective variety $X$ such that \eqref{snonempty} is satisfied. We start by introducing the notation which we will use to denote the centres of the blow-ups (we will use this notation in Section \ref{sec:cohomssneqspart2} as well).  

\begin{notation} \thlabel{cmaxnotation} 
Given the action of a group $H$ on a variety $Y$, for each $d \in \mathbb{N}$ we define $$C_d(Y,H) : = \{ y \in Y | \operatorname{dim} \operatorname{Stab}_H (y) = d\}$$ 
and $$d_{\operatorname{max}}(Y,H) : = \operatorname{max} \{\operatorname{dim} \operatorname{Stab}_{H} (y) \ | \ x \in Y \}.$$ Moreover, we let $C_{\operatorname{max}}(Y,H): = C_{d_{\operatorname{max}}(Y,H)}(Y,H)$ to simplify notation. 
\end{notation}

\begin{rk}[Closedness of $C_{\operatorname{max}}(\xmino,\hU)$ in $\xmino$]
Note that $C_{\operatorname{max}}(\xmino,\hU)$ is closed in $\xmino$, by standard result regarding upper semi-continuity of dimensions (in this case applied to stabiliser dimension), refer to \cite[\S 13.1]{Grothendieck1960-1961a} for example. 
\end{rk} 

The first step of the blow-up construction is to blow $X$ up along the closure of $C_{\operatorname{max}}(\xmino , \hU)$ in $X$. The key result is that the maximal dimension of stabiliser groups for points in the $\xmino$ for the blown-up space is strictly smaller than that for $X$ (see \cite[Prop 8.8]{Berczi2020}\footnote{The blow-up construction used in \cite{Berczi2020} consists in blowing $X$ up along the closure of $C_{\operatorname{max}}(\xmino , U)$ rather than that of $C_{\operatorname{max}}(\xmino , \hU)$. Nevertheless the results obtained for the former construction apply to the latter as well.}). Thus by repeating this procedure finitely many times we obtain a variety $\hX$ for which \eqref{ss=s(U)} is satisfied.

\begin{figure}[h]
\begin{subfigure}{1 \textwidth}
\centering
		\begin{tikzpicture}[scale=0.9,x=(15:6.5cm),y=(0:4cm), z=(90:1.8cm),
    cross line/.style={preaction={draw=white,-,line width=3pt}},
    equal/.style={double distance=1pt}]
    

\node (c1) at (1,-1.3,-0.25) {  } ;
\node (c2) at (1,-1.3,1.6) { };
\node (c3) at (1,0.35,-0.25) { };
\node (c4) at (1,0.35,1.6) { };
\node (c5) at (0,-1.3,-0.25) { } ;
\node (c6) at (0,-1.3,1.6) { } ;
\node (c7) at (0,0.35,-0.25) { } ;
\node (c8) at (0,0.35,1.6) { } ;

\node (1) at (0,-1,1.4) {\( X \)};
\node (1') at (0,0,1.4) {\( X \gitq G \)};
\node (2) at (0,-1,0.7) {\( X^{ss} \)};
\node (2') at (0,-1,0) {\( X^{s}  \)};
\node (3) at (0,0,0) {\( X^{s} / G \)};
\node (16) at (1,0,1.4) {\(   \tX \gitq G  \)};
\node(19) at (1,-1,1.4) {\( \tX \)};
\node(19') at (1,0,0.7) {\( \tX^{s} / G \)};
\node(20) at (1,-1,0.7) {\(  \tX^{ss} = \tX^{s}  \)};
\node(21) at (1,-1,0) {\( \tX^{ss} \setminus \widetilde{E} \)}; 
\node(22) at (1,0,0) {\( \left(\tX^{ss} \setminus \widetilde{E} \right) / G \)}; 
     
\filldraw  [fill = orange, draw=black,loosely dotted, opacity=0.2] (c5) rectangle (c8);
\filldraw  [fill = yellow, draw=black,loosely dotted, opacity=0.2] (c1) rectangle (c4);
  
\draw [black,loosely dotted] (c1) rectangle (c4);
\draw  [black,loosely dotted] (c5) rectangle (c8);
\draw [black, loosely dotted] (c3) -- (c7);
\draw [black, loosely dotted] (c4) -- (c8);
\draw [black, loosely dotted] (c1) -- (c5); 
\draw [black, loosely dotted] (c2) -- (c6);
     
\draw [->, dashed] (1) -- (1');
\draw [right hook ->](2) -- (1);
\draw [->] (2) -- (1');
\draw [right hook ->] (2') -- (2);
\draw [ ->](2') -- (3);
\draw [->] (21) -- (2') node[pos=0.5,below] {$\cong$}; 
\draw [right hook ->] (21) -- (20);
\draw [->] (21) -- (22) ;
\draw [->] (22) -- (3) node[pos=0.5,below] {$\cong$} ;
\draw [right hook -> ] (22) -- (19');
\draw [dashed, ->] (19) -- (16);
\draw[->] (20) -- (19');
\draw [equal] (19') -- (16);
\draw [right hook ->] (20) -- (19);
\draw [->] (19) -- (1) node[pos=0.5,above] {$\widetilde{\pi}$} ;
 
\draw [ForestGreen,line width=0.5mm, right hook ->, cross line] (3) -- (16);
\draw [ForestGreen, line width=0.5mm,right hook ->, cross line ] (3) -- (1');
    
		\end{tikzpicture}

	\caption{The partial desingularisation construction of classical GIT, when the condition that $X^{ss} = X^{s}$ is not satisfied for the linear action of a reductive group $G$ on a projective variety $X$. The variety $\tX$ is a blow-up of $X$, with exceptional divisor denoted by $\widetilde{E}$, and has an induced linear action of $G$.
} \label{fig:comparisonssneqsred} 
		
\end{subfigure}

\begin{subfigure}{1 \textwidth}

  \centering
       \begin{tikzpicture}[scale=0.85,x=(15:7cm),y=(0:4.3cm), z=(90:1.2cm),
    cross line/.style={preaction={draw=white,-,line width=3pt}},
    equal/.style={double distance=1pt}]
    
\node (c1) at (1,-1.3,-0.4) {  } ;
\node (c2) at (1,-1.3,2.4) { };
\node (c3) at (1,0.35,-0.4) { };
\node (c4) at (1,0.35,2.4) { };
\node (c5) at (0,-1.3,-0.4) { } ;
\node (c6) at (0,-1.3,2.4) { } ;
\node (c7) at (0,0.35,-0.4) { } ;
\node (c8) at (0,0.35,2.4) { } ;

\node (1) at (0,-1,2) {\( X \)};
\node (2') at (0,-1,0) {\( X^{\widehat{s}}  \)};
\node (3) at (0,0,0) {\( X^{\widehat{s}} / 
\hU \)};
\node (16) at (1,0,2) {\(   \hX \gitq \hU \)};
\node(19) at (1,-1,2) {\( \hX \)};
\node(20) at (1,-1,1) {\(   \hX^{s}   \)};
\node (19') at (1,0,1) {\( \hX^{s}  / \hU \)};
\node (21) at (1,-1,0) {\( \hX^{s} \setminus \widehat{E} \)}; 
\node (22) at (1,0,0) {\( \left(\hX^{s} \setminus \widehat{E} \right) / \hU \)}; 
     
\filldraw  [fill = red, draw=black,loosely dotted, opacity=0.2] (c5) rectangle (c8);
\filldraw  [fill = yellow, draw=black,loosely dotted, opacity=0.2] (c1) rectangle (c4);
  
\draw [black,loosely dotted] (c1) rectangle (c4);
\draw  [black,loosely dotted] (c5) rectangle (c8);
\draw [black, loosely dotted] (c3) -- (c7);
\draw [black, loosely dotted] (c4) -- (c8);
\draw [black, loosely dotted] (c1) -- (c5); 
\draw [black, loosely dotted] (c2) -- (c6);

\draw [equal] (16) -- (19');
\draw [right hook ->] (2') -- (1);
\draw [dashed, ->] (19) -- (16);
\draw[right hook ->] (20) -- (19);
\draw [->] (20) -- (19');
\draw [right hook ->] (21) -- (20);
\draw [->] (21) -- (22);
\draw [right hook ->] (22) -- (19') ;
\draw  [->] (21) -- (2') 
node[midway, below] {$\cong$};
\draw [->] (22) -- (3) 
node[midway, below] {$\cong$}; 
\draw [->] (2') -- (3);
\draw [->] (19) -- (1) node[midway,above] {$\widehat{\pi}$} ;
      \draw [ForestGreen,line width=0.5mm,right hook ->, cross line] (3) -- (16);
   		\end{tikzpicture} \caption{The $\hU$-theorem with blow-ups, when \eqref{ss=s(U)} is not satisfied for the linear action of an externally graded unipotent group $\hU$ on an irreducible projective variety $X$. The variety $\hX$ is obtained through a sequence of $\hU$-equivariant blow-ups of $X$, with exceptional divisor denoted by $\widehat{E}$. The green arrow denotes a projective completion} \label{fig:comparisonssneqsuhat}
            \end{subfigure}
            \caption{Comparison of GIT for reductive and for externally graded unipotent groups when the condition that `semistability coincides with stability' is not satisfied.} 
            \label{fig:comparisonssneqs}
       \end{figure}
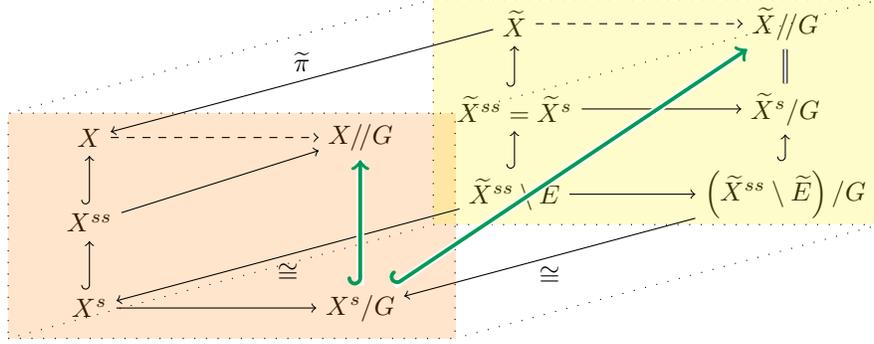
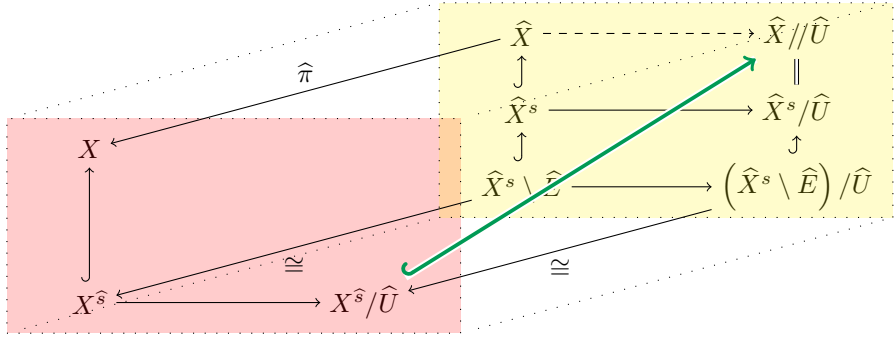

\subsubsection*{The blow-up construction for 
groups with internally graded unipotent radical.} The blow-up construction of \thref{uhatwbups} can be applied to enable the construction of quotients by linear algebraic groups $H$ with internally graded unipotent radical when \eqref{ss=s(U)} is not satisfied, by the method of quotienting in stages (we let $\hU$ denote the semi-direct product $U \rtimes \lambda(\GG_m)$ where $\lambda$ is the grading one-parameter subgroup): first by applying the $\hU$-theorem with blow-ups to the action of $\hU \subseteq H$ on $X$ (note that the blow-ups are not just $\hU$-equivariant but also $H$-equivariant, a necessary condition for the construction to apply in this more general case) and then by applying classical GIT to the action of $R/\lambda(\GG_m)$ on $\hX \gitq \hU$ (noting that this quotient has a suitable induced linear action of $R/\lambda(\GG_m)$).

\section{Smoothness of non-reductive fixed point sets} \label{sec:smoothness}

In this section we prove \thref{mainsmoothnessresult}. 
 We proceed in three steps. In Section \ref{subsec:basecase} we prove the result in the special case where $H  = \GG_a \rtimes \GG_m$ (see \thref{induction}). In Section \ref{subsec:Uhat} we extend the proof of \thref{induction} to establish the result in the case where $H = \hU $ is an externally graded unipotent group (see \thref{proofinspecialcase}). Finally in Section \ref{subsec:generalH} we show how the result for general $H$ follows from \thref{mainsmoothnessresult} and from known results regarding smoothness in the reductive case (see \thref{mainsmoothnessresultbody}). 

\subsection{The simplest non-reductive case: when $H = \GG_a \rtimes \GG_m$} \label{subsec:basecase} 

In this section we prove \thref{mainsmoothnessresult} in the special case where $H = \GG_a \rtimes \GG_m$. Although the proof we give in Section \ref{subsec:Uhat} of the result when $H  = \hU$ does not rely on the result in the special case, in the sense that we do not prove prove the more general result by induction on the dimension of $U$, the proof is nevertheless a direct generalisation of the proof in the special case. In particular all of the key ideas are already contained in the latter. For this reason we have chosen to present it separately in this section. The result in this special case is the following

\begin{theorem}[Smoothness of fixed point sets when $H = \GG_a \rtimes \GG_m$] \thlabel{induction}
Let $\hU = \GG_a \rtimes \GG_m$ where $\GG_m$ acts on $\operatorname{Lie} \GG_a$ with a single weight. Suppose that $\hU$ acts linearly on an irreducible projective variety $X$. Then $X^{\hU} \cap \xmino$ is smooth at any point at which $X$ is smooth.  In particular, if $X$ is smooth then $X^{\hU} \cap \xmino$ is smooth.\end{theorem} 

Our proof of \thref{induction} is based on an idea at the heart of Non-Reductive GIT, that of reducing to the case of classical GIT. Applied in the present context, this idea consists in proving the smoothness of $X^{\hU} \cap \xmino$ by reducing to showing that a reductive fixed point set is smooth.
To do so, we will construct from $X$ an auxiliary variety $Y$ with an action by a torus $T$ (hence reductive), with the property that each point $x \in X^{\hU} \cap \xmino$ has an associated point $y \in Y^T$ such that if $X$ is smooth at $x$, then $Y$ is smooth at $y$, and such that smoothness of $Y^T$ at $y$ implies smoothness of $X^{\hU} \cap \xmino$ at $x$. These implications will suffice to prove that if $X$ is smooth 
at $x \in X^{\hU} \cap \xmino$ then $X^{\hU } \cap \xmino$ is smooth at $x$. 

The construction of $Y$ which we will give in the proof of \thref{induction} relies on considering the following representations of $\hU$ and of $\GL(2;k)$. 

\subsubsection*{Representations of $\hU$ and of $\GL(2;k)$.} \label{setupforfinitecover}
Let $\hU : = \GG_a \rtimes \GG_m$ where $\GG_m$ acts with a positive weight $w$ on $\operatorname{Lie} \GG_a$ via the adjoint action. The coadjoint action of $\hU$ on $(\operatorname{Lie} \hU)^{\vee}$ gives a representation $\rho: \hU \to \operatorname{GL}((\operatorname{Lie} \hU)^{\vee})$. The restriction of $\rho$ to $\GG_a$ is an embedding because of the positive grading of $\GG_m$ on $\operatorname{Lie} \GG_a$. Moreover, we can choose a basis for $(\operatorname{Lie} \hU)^{\vee}$ to obtain an isomorphism $\operatorname{GL}((\operatorname{Lie} \hU)^{\vee}) \cong \GL(2;k)$ such that for $(u,t) \in \hU$, we have that $$\rho(u,t) = \begin{pmatrix} 
1 & u \\
0 & t^{-w}
\end{pmatrix}.$$ Here we have used that $\GG_m$ acts trivially on $(\operatorname{Lie} \GG_m)^{\vee}$ since $\GG_m$ is abelian. 

Let $\sigma: \GL(2) \to \GL(\operatorname{Sym}^{2} (k^2)) $ denote the standard representation of $\GL(2;k)$ on $\operatorname{Sym}^{2} (k^2)$, which under a suitable identification of $\GL(\operatorname{Sym}^{2} (k^2))$ with $\GL(3;k)$ is given by \begin{equation} \sigma \left(
 \begin{pmatrix} 
a & b \\
c & d 
\end{pmatrix}\right)  := \begin{pmatrix}  a^2 & 2 ab & b^2 \\
ac & bc+ad & bd \\
c^2 & 2 cd & d^2
\end{pmatrix}. \label{gl2rep}
\end{equation}
We let $\widetilde{\sigma}: \operatorname{Mat}_{2 \times 2}(k) \to \operatorname{Mat}_{3 \times 3}(k)$ denote the natural extension of this map. 

The composition $\sigma \circ \rho$ gives a representation of $\hU$ on $\operatorname{Sym}^{2} (k^2)$, and we consider the representation obtained by twisting the representation $\sigma \circ \rho$ by the character of $\rho(\hU)$ corresponding to the restriction of the determinant character of $\GL(2;k)$. This ensures that the image of $\sigma \circ \rho$ lies in $\SL(3;k)$ after twisting. We let $\widetilde{\rho}$ denote the resulting representation of $\hU$ on $\operatorname{Sym}^2 (k^2) \cong k^3$; it is given by \begin{equation}
\widetilde{\rho}
 (u,t) = \begin{pmatrix}   t^w & 2 u t^w & u^2 t^w \\
0 & 1 & u \\
0 & 0 & t^{-w} \end{pmatrix}. \label{uhatrep}
\end{equation}

The above representations play a key role in defining the auxiliary variety $Y$ constructed from $X$ and used to reduce the non-reductive fixed point set in $X$ to a reductive fixed point set in $Y$. 

\subsubsection*{The auxiliary variety $Y$.} The auxiliary variety $Y$ is constructed as a non-reductive GIT quotient for the action of $\hU$ on the product $X'$ of $X$ with a projective variety $W$ admitting a $\hU$-action which is defined as follows. 

The representation $\widetilde{\rho}: \hU \to \operatorname{GL}(\operatorname{Sym}^2(k^2)) \cong \operatorname{GL}(3;k)$ defined in \eqref{uhatrep} induces an action of $\hU$ on the vector space $\operatorname{End} ( \operatorname{Sym}^2 (k^2))$, which we identify with $\operatorname{Mat}_{3 \times 3}(k)$: the action is given by $(u,t) \cdot M = M \widetilde{\rho}(u,t)^{-1}$ for any $(u,t) \in \hU$ and any matrix $M \in V$. 

The projective completion $\mathbb{P}(\operatorname{Mat}_{3 \times 3}(k) \oplus k)$ of $\operatorname{Mat}_{3 \times 3}(k) $ admits a linear action of $\hU$ given by $(u,t) \cdot [M:v] = [M \widetilde{\rho}(u,t)^{-1}: t v ]$ for any $(u,t) \in \hU$ and $[M:v ] \in \mathbb{P}(\operatorname{Mat}_{3 \times 3}(k) \oplus k)$. We consider the closed and smooth subvariety $W$ of $\mathbb{P}(\operatorname{Mat}_{3 \times 3}(k) \oplus k)$ defined by restricting $\operatorname{Mat}_{3 \times 3}(k)$ to its subvariety consisting of matrices of the form \begin{equation} M = \begin{pmatrix} a^2 & 2 ab & b^2 \\
0 & 0 & 0 \\
0 & 0 & 0 \\
\end{pmatrix}. \label{choiceofM}
\end{equation} Note that such matrices represent the image under $\widetilde{\sigma}: \operatorname{Mat}_{2 \times 2} (k) \to \operatorname{Mat}_{3 \times 3} (k)$, the extension of the representation $\sigma: \GL(2;k) \to \GL(3;k)$ defined at \eqref{uhatrep}, of the subset of matrices with vanishing bottom row. 
This closed subvariety $W$ of $\mathbb{P}(\operatorname{Mat}_{3 \times 3}(k) \oplus k)$ is $\hU$-invariant and thus has an induced linear action of $\hU$ obtained by restricting that on $\mathbb{P}(\operatorname{Mat}_{3 \times 3}(k) \oplus k)$.

Now let $X' := W \times X.$ Then $X'$ has a natural $\hU$-action induced by that on $W$ and on $X$, and we can consider the linearisation of this $\hU$-action given by taking the tensor product of the pull-back to $X'$ of the linearisation on $W$ with the pull-back to $X'$ of the linearisation on $X$. We wish to define $Y$ as the non-reductive GIT quotient $Y \gitq \hU$ associated to the linear action of $\hU$ on $X' : = W \times X$ induced by that on $W$ and on $X$. By the results of Section \ref{sec:reviewofnrgit}, for this quotient to be well-defined we must have that the condition \eqref{ss=s(U)} is satisfied. The following lemma shows that this is the case.

\begin{lemma}[The quotient $Y : = X'\gitq \hU$ is well-defined] \thlabel{ss=sissatisfied}
The linear action of $\hU$ on $X' = W \times X$ induced by the linear action of $\hU$ on $W$ and $X$ satisfies the condition \eqref{ss=s(U)}. 
\end{lemma} 

\begin{proof}
Let $Z_{\operatorname{min}}'$ 
denote the analogue for $X'$ of $\zmin$. To show that \eqref{ss=s(U)} is satisfied, we must show that $\operatorname{Stab}_U((w,x)) = \{e\}$ for any $(w,x) \in \zmin'$. We will show instead that $\operatorname{Stab}_U (w) = \{e\}$ for ever $w \in Z(W)_{\operatorname{min}}$, which is sufficient since $Z'_{\operatorname{min}} = Z(W)_{\operatorname{min}} \times \zmin$. 

By definition of the action of the grading $\GG_m \subseteq \hU$, we have that
 \begin{equation} 
 Z(W)_{\operatorname{min}} = \left\{ \left.  \left[ \begin{pmatrix} 
a^2 & 0 & 0 \\
0 & 0 & 0\\
0 & 0 & 0  
\end{pmatrix} : 0 \right]  \  \right| \ a \in  k^{\ast} \right\}.
\label{descripofzmin}
\end{equation}
Moreover, it is easy to check that points in $Z(W)_{\operatorname{min}}$ have trivial unipotent stabiliser groups. Thus the condition \eqref{ss=s(U)} is satisfied for the linear action of $\hU$ on $X'$. 
\end{proof}

By \thref{ss=sissatisfied} we have that \eqref{ss=s(U)} is satisfied for the linear action of $\hU$ on $X'$. Thus by \thref{uhatthm}, after twisting the linearisation of the $\hU$-action on $X'$ by a suitable character and taking a sufficiently large tensor power, we obtain a projective geometric quotient $$\pi: X'^s \to Y: = X' \gitq \hU$$ where $X'^s = X'^0_{\operatorname{min}} \setminus U Z'_{\operatorname{min}}$ by \thref{uhatthm}. 

By its construction the variety $Y$ admits an action of $\GL(2;k)$ given by $$A \cdot \hU ([M:v],x) = \hU ([\sigma(A)M:v],x)$$ for every $A \in \GL(2;k)$ and for every $\hU$-orbit $\hU([M:v],x)$ of a point $([M:v],x) \in X'^0_{\operatorname{min}} \setminus U Z_{\operatorname{min}}'$ (we recall that $\sigma: \GL(2;k) \to \GL(3;k)$ is the map defined at \eqref{gl2rep}). In particular, we can consider the restricted action of the maximal torus $T \subseteq \GL(2;k)$ consisting of diagonal matrices.

\subsubsection*{Relating $\hU$-fixed points in $X$ to $T$-fixed points in $Y$.} 

As noted above, the key to proving \thref{induction} is relating $\hU$-fixed points in $X$ to $T$-fixed points in $Y$, so that the known result regarding smoothness for reductive group actions can be applied. The relationship we establish is given in the following 

\begin{lemma}[Relating $\hU$-fixed points in $X$ to $T$-fixed points in $Y$] \thlabel{relatingfixedpoints} 
Let $z \in \zmin$ and consider a point $([M:v],z) \in X' = W \times X$ where \begin{equation*} M = \begin{pmatrix} a^2 & 2 ab & b^2 \\
0 & 0 & 0 \\
0 & 0 & 0 \\
\end{pmatrix}
\end{equation*} 
for some non-zero $a$ and $b$ and $v$ is a non-zero scalar. Then $x' := ([M:v],z)$ lies in the $\hU$-semistable locus $X'^{s}$ so that the $\hU$-orbit $y : = \hU \cdot ([M:v],x)$ is a well-defined point of $Y$. Moreover, the point $y \in Y$ is fixed by $T$ if and only if $z \in \zmin$ is fixed by $\GG_a$. 
\end{lemma} 

\begin{proof} 
The non-zero assumption on $b$ ensures that the point $([M:v],z)$ lies in the stable locus $X'^0_{\operatorname{min}} \setminus U \zmin'$ for the $\hU$-action on $X'$ since $M \notin U Z(W)_{\operatorname{min}}$ yet $z \in \zmin$. Thus the $\hU$-orbit $y = \hU ([M:v],x)$ is a well-defined point in $Y$.  

By definition of the action of $\GL(2;k)$ on $Y$ and of the action of $\hU$ on $X$, we have that $y$ is fixed by $T$ if and only if for every $t_1,t_2 \in \GG_m$ there exists an element $(u,t)^{-1} \in \hU$ such that $$\left[   \begin{pmatrix} t_1^2 & 0 & 0 \\
0 & t_1 t_2 & 0 \\
0 & 0 & t_2^2  \end{pmatrix}  \begin{pmatrix} a^2 & 2ab & b^2 \\
0 & 0 & 0 \\
0 & 0 & 0 
\end{pmatrix} : v  \right]   = \left[   \begin{pmatrix} a^2 & 2ab & b^2 \\
0 & 0 & 0 \\
0 & 0 & 0 
\end{pmatrix} \begin{pmatrix}   t^w & 2 u t^w & u^2 t^w \\
0 & 1 & u \\
0 & 0 & t^{-w} \end{pmatrix} : t^{-1} v  \right] $$ and $$(u,t)^{-1} \cdot x  = x.$$
In other words, and using the assumption that $v \neq 0$, we have that $y \in Y^T$ if and only if for every $t_1,t_2 \in \GG_m$ there exists an element $(u,t)^{-1} \in \operatorname{Stab}_{\hU}(x)$ such that $$\begin{pmatrix} 
t_1^2 a^2 & t_1^2 2 ab & t_1^2 b^2 \\
0 & 0 & 0 \\
0 & 0 & 0
\end{pmatrix} =  \begin{pmatrix} 
t^{w+1} a^2 & 2 u t^{w+1} a^2 + 2 ab  t  & u^2 t^{w+1} a^2 + 2 ab u t  + t^{-w+1} b^2  \\
0 & 0 & 0 \\
0 & 0  & 0
\end{pmatrix}.$$ 
Equality of the first two entries of the above matrices implies that $(u,t) \in \operatorname{Stab}_{\hU} (z)$ must satisfy $t^{w+1} = t_1^2$ and $u = b (t_1^2 - t)/a t_1^2.$ Note that equality of the first two entries ensures equality of the third since they are matrices of the form given in \eqref{choiceofM}. 

From the above calculation we obtain that $y = \hU \cdot ([M:v],z)$ lies in $Y^T$ if and only if $(u,t)^{-1} \in \operatorname{Stab}_{\hU}(z)$ for every $(u,t) \in \hU$ such that $u = b(t^w-1)/at^w$. The latter condition implies that $z$ is fixed by $U=\GG_a$ (since $z$ is already fixed by the grading $\GG_m$) and hence that $z \in X^{\GG_a} \cap \zmin$. Conversely, if $z \in X^{\GG_a} \cap \zmin$, then $(u,t) \in \operatorname{Stab}_U(z)$ for any $(u,t)$ satisfying $u = b(t^w-1)/at^w$. We have thus established that $y =\hU \cdot ([M:v],z) \in Y$ is fixed by $T$ if and only if $z \in \zmin$ is fixed by $\GG_a$. 
\end{proof} 

We can now prove \thref{induction}.

\begin{proof}[Proof of \thref{induction}]
Let $z_0 \in X^{\hU} \cap \xmino$ and suppose that $X$ is smooth at $z_0$. We wish to show that $X^{\hU} \cap \xmino$ is smooth at $z_0$, and we do so by using the auxiliary variety $Y$ admitting the action of the maximal torus $T \subseteq \GL(2;k)$. 

To this end, let $y_0 : = \hU \cdot ([M_0:1],z_0) \in Y$ where $M_0$ is chosen as at \eqref{choiceofM}, with $a$ and $b$ non-zero. Then by \thref{relatingfixedpoints}, we have that $y_0 \in Y^T$. Our aim is to show that $Y^T$ is smooth at $y_0$ by using the classical result concerning smoothness of reductive fixed point sets (see \cite[Prop 1.3]{Iversen1972}). To do so we must first establish that $Y$ is smooth at $y_0$. Since $X$ is smooth at $z_0$ by assumption and since $W$ is smooth at $M_0$, we have that $X'$ is smooth at $([M_0:1],z_0)$. Moreover, the point $([M_0:1],z_0)$ has trivial stabiliser group in $\hU$ since $\GG_m$ acts non-trivially on the scalar coordinate. Since by construction $Y$ is a geometric $\hU$-quotient, it follows that $Y$ is smooth at $y_0$ (we note that in general $Y$ will be smooth at the $\hU$-orbit of any point $([M:v],z) \in X'^{s,\hU}$ such that $M$ is of the form given in \eqref{choiceofM} with $a$ and $b$ non-zero and $v$ is non-zero). Therefore $Y$ is smooth in a small neighbourhood of $y_0$. It follows from the reductivity of $T$ that $Y^T$ must therefore be smooth at $y_0$. Since smoothness is an open condition, we have that $Y^T$ is also smooth in an open neighbourhood of $y_0$ contained in $Y^T$, and we call this neigbhourhood $N_{y_0}$.

The preimage of $N_{y_0}$ is an open neigbhourhood in $X'^0_{\operatorname{min}} \setminus U \zmin'$ of $([M_0:1],z)$, or equivalently a product of open neighbourhoods $N_{[M_0:1]}$ of $[M_0:1]$ and $N_{z_0}$ of $z_0$ respectively. By making $N_{y_0}$ smaller if necessary, we can assume that $N_{[M_0:1]}$ is contained in the open subvariety of $W$ consisting of points $[M:v]$ where $v$ is non-zero. Note that $N_{[M_0:1]}$ is then contained as an open subset in $W^0_{\operatorname{min}} \setminus U Z(W)_{\operatorname{min}}$. Since the action of $\hU$ on the projective variety $W$ satisfies \eqref{ss=s(U)}, thus giving the existence of a projective geometric quotient for the action of $\hU$ on $W^0_{\operatorname{min}} \setminus U Z(W)_{\operatorname{min}}$, it follows that $N_{[M_0:1]}$ admits a geometric $\hU$-quotient $N_{[M_0:1]} /\hU$.

For the neigbhourhood $N_{z_0}$, we use again \thref{relatingfixedpoints}. That is, by making $N_{y_0}$ yet smaller if necessary, we can ensure that $N_{z_0}$ is contained in $X^{\GG_a} \cap \zmin$. The inclusion of the open neighbourhood $N_{z_0} \subseteq X^{\GG_a} \cap \zmin$ shows that there is a well-defined map $$N_{y_0} \to N_{[M_0:1]}/ \hU \times N_{z_0}$$ given by $\hU \cdot ([M:v],z) \mapsto  (\hU \cdot ([M:v]), z)$, which is in fact an isomorphism. 

Since we know that $N_{y_0}$ is smooth, it follows that $N_{z_0} \subseteq X^{\GG_a} \cap \zmin$ must also be smooth. This establishes that $X^{\GG_a} \cap \zmin$ is smooth at $z_0$. 
\end{proof}

\subsection{The externally graded unipotent case} \label{subsec:Uhat} 

In this section we generalise \thref{induction} to the case where the group is an externally graded unipotent group $\hU: = U \rtimes \GG_m$ with $U$ abelian and $\GG_m$ acting with a single weight on $\operatorname{Lie}U$. That is, we prove the following 

\begin{theorem} \thlabel{proofinspecialcase}
Let $\hU: = U \rtimes \GG_m$ where $U$ is an abelian unipotent group and $\GG_m$ acts with a single and positive weight on $\operatorname{Lie} U$ via the adjoint action. Suppose that $\hU$ acts linearly on an irreducible projective variety $X$ and let $x \in X^{\hU} \cap \xmino$. Then the subvariety $X^{\hU} \cap \xmino$ is smooth at $x$ if $X$ is smooth at $x$. 
\end{theorem} 

Our proof of this result is an extension of the proof of \thref{induction} to the case where $\operatorname{dim} U \geq 1$ and we follow the steps used to prove this result in Section \ref{subsec:basecase}. The first step is to define an auxiliary variety $Y_d$ admitting a torus action.

\subsubsection*{The auxiliary variety $Y_d$ when $\operatorname{dim} U = d\geq 1$.} 

As in Section \ref{subsec:basecase}, we let $W$ denote the closed subvariety of $\mathbb{P}(\operatorname{Mat}_{3 \times 3}(k)  \oplus k)$ consisting of pairs $[M:v]$ where $M$ is of the form given in \eqref{choiceofM} with $a,b \in k$. 

Let $d = \operatorname{dim} U$. Then since $U$ is abelian we can identify $U$ with $\GG_a^{d}$. Under this identification, we can consider the embedding $\hU \cong \GG_a^{d} \rtimes \GG_m \hookrightarrow \widehat{\GG}_a \times \cdots \times \widehat{\GG}_a$ given by $(u_1,\hdots,u_d,t) \mapsto ((u_1,t),\hdots,(u_d,t)).$ We identify the product $\widehat{\GG}_a \times \cdots \times \widehat{\GG}_a$ with $\hU_T : = U \rtimes T$ where $T = \GG_m^d$. 

We now consider the product $W^d$ of $d$ copies of $W$, on which we define a linear action of $\widehat{\GG}_a \times \cdots \times \widehat{\GG}_a$ as follows: the element $((u_1,t_1),\hdots,(u_d,t_d))$ acts on the $i$-th factor of $W^d$ via $(u_i,t_i)$ via multiplication on the right by $\widetilde{\rho}(u_i,t_i)^{-1}$ on matrices (see \eqref{uhatrep} for the definition of $\widetilde{\rho}$), and by multiplication by $t_i$ on the scalar coordinate. 

Each copy of $\widehat{\GG}_a$ inside the product $\widehat{\GG}_a \times \cdots \times \widehat{\GG}_a$ acts on $X$ via the restriction to $\widehat{\GG}_a \subseteq \hU$ of the $\hU$-action on $X$; we note that a point in $X$ is fixed by $\hU$ if and only if it is fixed by $\widehat{\GG}_a \times \cdots \times \widehat{\GG}_a$. We consider the natural linear action of $\widehat{\GG}_a \times \cdots \times \widehat{\GG}_a$  on the product $X' = W^d \times X$ induced by the linear action on each factor. Just as in the $d=1$ case, we obtain that the condition \eqref{ss=s(U)} is satisfied for this action (see \thref{ss=sissatisfied}). Thus we have a projective geometric quotient $$\pi: {X'}^{ss} \to  Y_d : = X' \gitq \widehat{U}_T.$$ 

We define an action of $(\GL(2;k))^d$ on $Y_d$, given by multiplication on the left for the matrices, and by the trivial action on the scalars and on $X$. We consider the restricted action of the maximal torus $T^d \subseteq (\GL(2;k)^d)$ consisting of $d$-tuples of diagonal matrices.

\subsubsection*{Relating $\hU$-fixed points in $X$ to $T^d$-fixed points in $Y_d$.}

We can now relate, as per the $d=1$ case, the condition of a point $z \in \zmin$ being fixed by $U$ to the condition of an associated point $y \in Y_d$ being fixed by $T^d$. That is, we prove the following generalisation of \thref{relatingfixedpoints}. 

\begin{lemma}[Relating $\hU$-fixed points in $X$ to $T^d$-fixed points in $Y_d$]  \thlabel{fixedpointcharac2}
Let $z \in \zmin$, let $$M = \begin{pmatrix} 
a^2 & 2 ab & b^2 \\
0 & 0 & 0 \\
0 & 0 & 0 
\end{pmatrix}$$ for some $a,b \neq 0$ and choose $v \neq 0$. Then $x':= ([M:v],\hdots,[M:v],z)$ lies in the $\hU_T$-semistable locus $X'^{ss}$, so that the $\hU$-orbit $y:= \hU_T ([M:v],\hdots,[M:v],z)$ is a well-defined point of $Y_d$. Moreover, the point $y \in Y_d$ is fixed by $T^d$ if and only if $z \in \zmin$ is fixed by $U$. 
\end{lemma} 

\begin{proof} 
To show that $x'$ lies in the $\hU_T$-semistable locus $X'^{ss}$, we can use the Hilbert-Mumford criterion for the action of groups of the form $\hU_T$, as stated in Section \ref{subsec:ss=snonred} (see \eqref{hmcriterionH}). That is, the semistable locus $X'^{ss}$ for the action of $\hU_T$ on $X'$ is given by \begin{equation} X'^{ss} = \bigcap_{u \in U} u X'^{ss,T}. \label{tss}
\end{equation}  

To describe the semistable locus for this torus action, we let $p_i$ denote the projection from $W^d \times X $ to the product of the $i$-th factor of $W^d$ with $X$, which has an action of $\hU$. We recall that the semistable locus for the action of the grading $\GG_m$ in $\hU$ on $W \times X$ (after having suitably modified the linearisation) is given by $(W^0_{\operatorname{min}} \times \xmino) \setminus Z(W \times X)_{\operatorname{min}}$. 
It follows that $$X'^{ss,T} = X'^0_{\operatorname{min}} \setminus \bigcup_{i=1}^d p_i^{-1}(Z(W \times X)_{\operatorname{min}}).$$ By \eqref{tss}, we obtain that $$X'^{ss} =  X'^0_{\operatorname{min}} \setminus \bigcup_{i=1}^d p_i^{-1}( \GG_a Z(W \times X)_{\operatorname{min}}).$$ 

Since $z \in \zmin$ and $[M{:}v] \notin Z(W)_{\operatorname{min}}$, we know that $([M{:}v],z) \notin \GG_a Z(W \times X)_{\operatorname{min}} = \GG_a (Z(W)_{\operatorname{min}} \times \zmin)$. 

By the same calculations as in the $d=1$ case for each factor $W $, we obtain that $y: = \hU_T \cdot ([M{:}v],\hdots,[M{:}v],z)$ is fixed by $T^d$ if and only if $z$ is fixed by $\widehat{\GG}_a \times \cdots \times \widehat{\GG}_a$ or equivalently by $U$. 
\end{proof} 

We now have all the ingredients needed to generalise the proof of \thref{induction} to the case where $\operatorname{dim} U \geq 1$.  

\begin{proof}[Proof of \thref{proofinspecialcase}]
Let $z_0 \in X^U \cap \zmin$ and suppose that $X$ is smooth at $z_0$. we wish to show that $X^U \cap \zmin$ is smooth at $z_0$ using the action of the torus $T^d$ on the auxiliary variety $Y_d$. To simplify notation we set $T = T_d$ and $Y = Y_d$. As in the $d=1$ case, we let $$M_0 = \begin{pmatrix} 
a^2 & 2 ab & b^2 \\
0 & 0 & 0 \\
0 & 0 & 0 
\end{pmatrix}$$ for some $a,b \neq 0$ and let $y_0: = \hU_T  \cdot ([M_0:1],\hdots, [M_0:1],z_0) \in Y$. Then by \thref{fixedpointcharac2} we have that $y_0 \in Y^{T}$. Just as in the proof of \thref{induction}, we have that $Y$ is smooth at $y_0$ since $\pi$ is a geometric quotient and $v_i \neq 0$ for all $i$. Thus since $T$ is reductive we can conclude that $Y^{T^d}$ is smooth at $y_0$, and also in a small neighbourhood $N_{y_0}$ of $y_0$. 

The preimage of $N_{y_0}$ under $\pi$ is an open neighbourhood of $([M_0:1],\hdots,[M_0:1],z_0)$ in $X'^{ss}$, or equivalently a product of open neigbhourhoods $N_{([M_0:1],\hdots,[M_0:1])}$ and $N_{z_0}$ of $([M_0:1],\hdots,[M_0:1])$ and of $z_0$ respectively. By \thref{fixedpointcharac2}, we know that by making $N_{y_0}$ smaller if necessary we can ensure that $N_{z_0}$ is contained in $X^U \cap \zmin$ so that it is an open neighbourhood of $z_0$ in $X^U \cap \zmin$. Then by using the same argument as in the last part of the proof of \thref{induction}, we can conclude that $N_{z_0}$ must also be smooth. This proves that $X^U \cap \zmin$ is smooth at $z_0$. 
\end{proof}

\subsection{The general non-reductive case}  \label{subsec:generalH}

We conclude this section by proving \thref{mainsmoothnessresult}, which is a simple consequence of \thref{proofinspecialcase}. 
 
\begin{cor}[\thref{mainsmoothnessresult}] \thlabel{mainsmoothnessresultbody} 
Let $H = U \rtimes R$ denote a linear algebraic group such that $U$ is abelian and such that $R$ contains a central one-parameter subgroup $\lambda: \GG_m \to Z(R)$ acting with a single and positive weight on the Lie algebra of $U$ via the adjoint action. Suppose that $H$ acts on an irreducible projective scheme $X$, and let $\xmino$ denote the open Bialynicki-Birula stratum associated to the action of $\lambda(\GG_m)$ on $X$. Then $X^H \cap \xmino$ is smooth at any point at which $\xmino$ is smooth. In particular, the scheme $X^H \cap \xmino$ is smooth if $X$ is smooth. 
\end{cor} 

\begin{proof}
Suppose that $X$ is smooth at a point $x \in X^H \cap \xmino$. Let $\hU := U \rtimes \lambda(\GG_m)$ and note that $x \in X^{\hU} \cap \xmino$. Then by \thref{proofinspecialcase}, we have that $X^{\hU} \cap \xmino$ is smooth at $x$. 
The variety $X^{\hU} \cap \xmino$ has an induced action of $H/\hU$, which is reductive. Thus by \cite[Thm 1.3]{Iversen1972} we obtain that $(X^{\hU} \cap \xmino)^{H/\hU}$ is smooth at $x$. By observing that $$(X^{\hU} \cap \xmino)^{H/\hU} = (X^{\hU})^{H/\hU} \cap \xmino = X^H \cap \xmino,$$ we can conclude that $X^H \cap \xmino$ is smooth at $x$ as required. 
\end{proof}

\section{Cohomology when `semistability coincides with stability'} 

\label{sec:cohomss=s}

Having proved \thref{mainsmoothnessresult} in Section \ref{sec:smoothness} above, our aim for the remainder of this paper is to prove the second of the two main results of this paper, namely \thref{maintheorem}. That is, we wish to compute the Poincar\'e series of non-reductive GIT quotients in the case where `semistability does not coincide with stability'. Doing so relies on existing results for computing the Poincar\'e series of GIT quotients (both classical and non-reductive) when `semistability coincides with stability'; the aim of this section is to summarise these results. In Section \ref{subsec:cohomred} we consider the reductive case, in Section \ref{subsec:cohomuhat} the externally graded unipotent case and finally in Section \ref{subsec:cohomH} the general internally graded non-reductive case. 

We assume from here on that we are working over the field of complex numbers $k = \CC$. 

\subsection{The reductive case} \label{subsec:cohomred}

Given the linear action of a reductive group $G$ on a smooth projective variety $X$, a method is introduced in \cite{Kirwan1984} for inductively computing the $G$-equivariant Poincar\'e series of the semistable locus $X^{ss}$. When semistability coincides with stability this method gives the Poincar\'e series of the GIT quotient $X \gitq G$. The method of \cite{Kirwan1984} relies on the GIT-instability stratification associated to the linear action of a reductive group $G$ on a projective variety $X$, and thus we start by reviewing its properties below.

\subsubsection*{GIT-instability stratification.} 
Given the choice of an invariant inner product on the Lie algebra of maximal torus of $G$, there is a finite stratification \begin{equation} X = \bigsqcup_{\beta \in \mathcal{B}} S_{\beta} \label{GITinstabstrat} \end{equation} of $X$ satisfying the following properties (see \cite[\S 12]{Kirwan1984} for the construction of the strata): 
\begin{enumerate}[(i)]
\item the open stratum $S_0$ coincides with $X^{ss}$;
\item  for each $\beta \neq 0$, there is an isomorphism \begin{equation} 
S_{\beta} \cong G \times_{P_{\beta}} Y_{\beta}^{ss}  \label{sbetadescrip}
\end{equation} where $Y_{\beta}^{ss}$ is a locally closed subvariety of $X$ and $P_{\beta}$ is a parabolic subgroup of $G$;
\item for every $\beta \neq 0$, there exists a $P_{\beta}$-equivariant locally trivially fibration $p_{\beta}: Y_{\beta}^{ss} \to Z_{\beta}^{ss}$ with affine spaces as fibres, where $Z_{\beta}^{ss}$ is the semistable locus for the linear action of a maximal reductive subgroup $\operatorname{Stab} \beta$ of $P_{\beta}$ on a closed subvariety $Z_{\beta}$ of $X$. \label{fibration}
\end{enumerate} 

\subsubsection*{Equivariant Poincar\'e series of the semistable locus.} 
It can be shown using techniques from symplectic geometry (see \cite[Thm 5.4]{Kirwan1984}) that the stratification \eqref{GITinstabstrat} is equivariantly perfect. By definition (see \cite[\S 2.16]{Kirwan1984}), this means that  \begin{equation} P_t^G(X) = P_t^{G}(X^{ss}) +  \sum_{\beta \in \mathcal{B} \setminus \{0\}} t^{2 d(\beta)} P_t^{G}(S_{\beta}), \label{eqperfect}
\end{equation} where $d(\beta)$ is the complex codimension of $S_{\beta}$ in $X$. For simplicity we have assumed that all of the connected components for a given $S_{\beta}$ have the same dimension (so that $d(\beta)$ is well-defined). If this is not the case, the formula needs modification (see \cite[\S 8.12]{Kirwan1984}). 

The Poincar\'e series $P_t^G(X)$ and $P_t^G(S_{\beta})$ can be further simplified. For the former, we can use the fact that if $X$ is a smooth projective variety acted upon by a reductive group $G$, then $P_t^G(X) = P_t(X) P_t(BG)$ (see \cite[Prop 5.8]{Kirwan1984}). For the latter, we can use the isomorphism $S_{\beta} \cong G \times_{P_{\beta}} Y_{\beta}^{ss}$ (see \eqref{sbetadescrip} above) and the map $p_{\beta}: Y_{\beta}^{ss} \to Z_{\beta}^{ss}$ introduced in \ref{fibration} above. Using moreover the fact that $P_{\beta}$ is homotopically equivalent to $\operatorname{Stab} \beta$, we obtain that for each stratum $S_{\beta}$ there is an isomorphism of rational cohomology groups $H_G^{\ast}(S_{\beta}, \mathbb{Q}) \cong H_{\operatorname{Stab} \beta}^{\ast} (Z_{\beta}^{ss}, \mathbb{Q})$ from which it follows that $P_t^G(S_{\beta}) = P_t^{\operatorname{Stab}\beta} (Z_{\beta}^{ss}).$
Thus we obtain an inductive formula for the $G$-equivariant Poincar\'e series of $X^{ss}$: \begin{equation} P_t^G(X^{ss}) = P_t(X) P_t(BG) + \sum_{\beta \in \mathcal{B} \setminus \{0\}} t^{2 d(\beta)} P_t^{\operatorname{Stab} \beta} (Z_{\beta}^{ss}). \label{formulass}
\end{equation}

\subsubsection*{Poincar\'e series of $X \gitq G$ when $X^{ss}=X^{s}$.} If moreover the semistable and stable loci coincide, the action of $G$ on $X^{ss}$ has at worst finite stabiliser groups so that the equivariant rational cohomology of the semistable locus coincides with the rational cohomology of the GIT quotient. Thus we have that \begin{equation} P_t(X \gitq G) = P_t^G(X^{ss})  = P_t(X) P_t(BG) + \sum_{\beta \in \mathcal{B} \setminus \{0\}} t^{2 d(\beta)} P_t^{\operatorname{Stab} \beta} (Z_{\beta}^{ss}). \label{redformula} 
\end{equation}

\subsection{The externally graded unipotent case} \label{subsec:cohomuhat}

A similar approach to that described in Section \ref{subsec:cohomred} 
is adopted in \cite{Berczi2019} to obtain a formula for the Poincar\'e series of non-reductive GIT quotients of the form $X \gitq \hU$ when \eqref{ss=s(U)} is satisfied. However, the situation is simpler in this case than in the reductive case: an explicit, rather than inductive, formula can be obtained thanks to the existence of a distinguished subvariety $\zmin$ of $X$ which carries cohomological information about $\xmino$. The approach of \cite{Berczi2019} can be summarised as follows. 

Suppose that an externally graded unipotent group $\hU$ acts linearly on a smooth projective variety $X$ such that the condition \eqref{ss=s(U)} is satisfied. Then by \thref{uhatthm} (the $\hU$-theorem), there exists a projective quotient $X \gitq \hU$ which is a geometric quotient for the action of $\hU$ on the semistable locus $X^{ss} = \xmino \setminus U \zmin$. By \cite[Cor 5.4]{Berczi2019}, the stratification \begin{equation} \xmino = X^{ss} \sqcup U \zmin  \label{perfstrat}
\end{equation} of $\xmino$ is equivariantly perfect so that $$P_t^{\hU} (\xmino) = P_t^{\hU}(X^{ss}) + t^{2d} P_t^{\hU}(U \zmin)$$ where $d = \operatorname{dim} X - \operatorname{dim} U - \operatorname{dim} \zmin$ is the complex codimension of $U\zmin$ in $\xmino$. Moreover, since the condition \eqref{ss=s(U)} is satisfied, the stabiliser groups of points in $X^{ss}$ are finite and thus as in the classical case we have that $P_t(X \gitq \hU) = P_t^{\hU}( X^{ss})$.

The fact that $\xmino$ retracts onto $\zmin$ and that $\hU$ is homotopically equivalent to the grading $\GG_m$ allows further simplification of the formula for $P_t^{\hU}( X^{ss})$. Indeed, it implies that $P_t^{\hU} (\xmino) = P_t^{\GG_m} (\zmin)$, and since $\GG_m$ acts trivially on $\zmin$, there is an equality $$P_t^{\GG_m}(\zmin)  = P_t(\zmin) P_t(B\GG_m) = P_t(\zmin) \frac{1}{1-t^2}.$$ Thus replacing and rearranging the terms of \eqref{perfstrat}, we obtain that \begin{equation} P_t(X \gitq \hU) = P_t(\zmin) \frac{1 - t^{2d} }{1-t^2}. \label{uhatformula}\end{equation}  This formula contrasts with \eqref{redformula} in the reductive case where an inductive procedure is needed to compute $P_t(X \gitq G) = P_t^G(X^{ss})$; for the $\hU$-action there is a distinguished subvariety $\zmin$ of $X$ which carries much of the cohomological information of $X \gitq \hU$.

\subsection{The general non-reductive case}
\label{subsec:cohomH} 

As we have seen in Section \ref{subsec:ss=snonred}, if $H = U \rtimes R$ is a linear algebraic group with internally unipotent radical $U$ acting linearly on an irreducible projective variety $X$, such that \eqref{ss=s(U)} is satisfied, then a GIT quotient $X \gitq H$ can be constructed in stages, first by quotienting by $\hU$, and then by quotienting by the residual reductive group $R_{\l}: =R /\lambda(\GG_m)$ where $\lambda(\GG_m) \subseteq Z(R)$ is the grading one-parameter subgroup. If we moreover assume that semistability coincides with stability for the action of $R_{\l}$ on $X \gitq \hU$, then the formulae \eqref{redformula} and \eqref{uhatformula} from the above Sections \ref{subsec:cohomred} and \ref{subsec:cohomuhat} for the Poincar\'e series of classical and $\hU$-GIT quotients can be combined to produce an inductive formula for the Poincar\'e series of GIT quotients by linear algebraic groups $H$ with internally graded unipotent radical.

However, the resulting formula involves the GIT-instability stratification for the action of $R_{\l}$ on the intermediate quotient $X \gitq \hU$ (the choice of an invariant inner product on $H$ induces one on $R_{\l}$), which may be difficult to describe in practice (for example if the intermediate quotient has no obvious modular interpretation, in the case where we are using GIT to construct a moduli space). For this reason, a different approach is given in \cite{Berczi2019} for computing the Poincar\'e series of the GIT quotient $X \gitq H$, resulting in a formula which depends on information only about $X$, rather than about the intermediate quotient $X \gitq \hU$ as well. This approach requires assuming that semistability coincides with stability for the action of $R_{\l}$ on $\zmin$, in addition to assuming that \eqref{ss=s(U)} is satisfied and that semistability coincides with stability for the action of $R_{\l}$ on $X \gitq \hU$. 

Under this assumption, by \cite[Lem 5.6]{Berczi2019} we have that $$X \gitq H = \left( p^{-1}(\zmin^{ss,R_{\l}}) \setminus U  \zmin^{ss,R_{\l}}  \right) / H.$$ The stratification approach from Section \ref{subsec:cohomuhat}, this time applied to $p^{-1}(\zmin^{ss,R_{\l}})$ instead of $\xmino$, can then be used to obtain that \begin{equation} 
P_t(X \gitq H)  = P_t(\zmin \gitq R_{\l}) \frac{1 - t^{2d}}{1-t^2}. \label{Hformula}
\end{equation} 
The Poincar\'e series $P_t(\zmin \gitq R_{\l})$ can in turn be computed using \eqref{redformula}, which involves the GIT-instability stratification for the action of $R_{\l}$ on $\zmin$, rather than a GIT-instability stratification on the intermediate quotient $X \gitq \hU$ which is typically harder to describe explicitly.

\section{Cohomology when `semistability does not coincide with stability'} 
\label{sec:cohomssneqspart2}
In this section we prove \thref{maintheorem}, which provides a method for computing the Poincar\'e series of non-reductive GIT quotients when `semistability does not coincide with stability' and when the internally graded group $H$ has an abelian unipotent radical with the grading $\GG_m$ acting with a single weight on $\operatorname{Lie} U$. That is, we study the Poincar\'e series of the non-reductive GIT quotient obtained after the sequence of non-reductive blow-ups. We recall that we are working over the field $k = \mathbb{C}$ of complex numbers.  

In Section \ref{subsec:nrsmoothbulocus} we prove that the centres of the blow-ups are smooth if the initial variety is smooth (see \thref{smoothnessofbulocus}), a necessary preliminary result for proving \thref{maintheorem}. In Section \ref{subsec:simplesnonred} we use this result to establish a formula for the Poincar\'e series of the quotient $\hX \gitq \hU$ by the externally graded unipotent group $\hU : = U \rtimes \lambda(\GG_m) \subseteq H$, establishing parts \ref{part1} and \ref{part2} of \thref{maintheorem} (see \thref{nonredformula}). In Section \ref{subsec:gennonred} we extend this formula to the Poincar\'e series of the quotient $\hX \gitq H$ under an additional assumption on the action of $H / \hU$ to prove part \ref{part3} of \thref{maintheorem} (see \thref{mostgeneralformula}).

\subsection{Smoothness of the centres of the blow-ups}  \label{subsec:nrsmoothbulocus}

In this section we prove that the centres of the blow-ups from Non-Reductive GIT are smooth if the initial variety is smooth. This result immediately implies (see \thref{Uabelian} below) that if $X$ is smooth then the open Bialynicki-Birula stratum $\hX^0_{\operatorname{min}}$ for the resulting variety is also smooth, and thus that the non-reductive GIT quotient $\hX \gitq \hU$ has at worst finite quotient singularities.

\begin{theorem}[The centres of the blow-ups are smooth if $X$ is smooth]\thlabel{smoothnessofbulocus} 
Let $\hU: = U \rtimes \GG_m$ where $U$ is an abelian unipotent group and $\GG_m$ acts on $\operatorname{Lie} U$ via the adjoint action with a strictly positive weight. Suppose that $\hU$ acts linearly on an irreducible projective variety $X$. 
Then $C_{\operatorname{max}}(X^0_{\operatorname{min}},\hU)$ is smooth at $x \in C_{\operatorname{max}}(X^0_{\operatorname{min}},\hU)$ if $X$ is smooth at $x$. 
In particular, if $X$ is smooth then $C_{\operatorname{max}}(X^0_{\operatorname{min}},\hU)$ is a smooth subvariety of $\xmino$.
\end{theorem}

\begin{cor}[If $X$ is smooth then $\hX \gitq \hU$ is cohomologically smooth when $U$ is abelian] \thlabel{Uabelian}
Let $\hU: = U \rtimes \GG_m$ where $U$ is an abelian unipotent group and $\GG_m$ acts on $\operatorname{Lie} U$ via the adjoint action with a strictly positive weight. Suppose that $\hU$ acts linearly on a smooth projective variety $X$ such that \eqref{snonempty} is satisfied. Then the projective geometric quotient $\hX \gitq \hU$ obtained from the blow-up construction of \thref{uhatwbups} has at worst finite quotient singularities. 
\end{cor}

We will prove \thref{smoothnessofbulocus} using \thref{mainsmoothnessresult}, by reducing to the case where the maximal dimension of unipotent stabiliser group for points in $\xmino$  coincides with the dimension of $U$. Indeed in this case $C_{\operatorname{max}}(\xmino,\hU) = X^U \cap \zmin$ which is smooth if $X$ is smooth by \thref{mainsmoothnessresult}.

\begin{proof}[Proof of \thref{smoothnessofbulocus}]
We first show that $$C_{\operatorname{max}}(\xmino,\hU) = C_{\operatorname{max}}(\xmino, U) \cap U \zmin. \label{usefuleq}$$ For the inclusion, we note that if $x \in C_{\operatorname{max}}(\xmino,\hU)$, then it must have the same unipotent stabiliser dimension as $p(x)$. This is because $\GG_m$ normalises $U$, which implies that $\operatorname{dim} \operatorname{Stab}_U(p(x)) \geq \operatorname{dim} \operatorname{Stab}_U(x)$ for every $x \in \xmino$ (see \cite[Rk 5.7]{Berczi2020}). Moreover, since $p(x) \in \zmin$ is fixed by $\GG_m$, then $x$ must also be fixed by a one-parameter subgroup of $\hU$, which will be a conjugate of the grading $\GG_m$ in $\hU$ and thus we obtain that $x \in U \zmin$. For the reverse inclusion, if $x  = u \cdot z \in  U \zmin$ has maximal dimension stabiliser group in $U$, then $ u^{-1} \cdot x = z$, so that $u^{-1} \cdot x \in \zmin$ and hence $x$ is fixed by $u^{-1}  \GG_m u^{-1}$. Thus $\operatorname{dim} \operatorname{Stab}_{\hU}(x) = \operatorname{dim} \operatorname{Stab}_U (x) +1$ and so $x$ lies in $C_{\operatorname{max}}(\xmino,\hU)$. 

Suppose that $X$ is smooth at a point $x_0 \in C_{\operatorname{max}}(\xmino, U) \cap U \zmin$. Our aim is to show that $ C_{\operatorname{max}}(\xmino, U) \cap U \zmin$ is smooth at $x_0$. Our proof of this result relies on the set-up established in \cite[\S 7.1, pp 33-34]{Berczi2020} to prove part of \thref{uhatthm}, namely to construct a locally trivial $U$-quotient of $\xmino$ with an explicit $\GG_m$-equivariant projective completion, given the action of an externally graded unipotent group $\hU$ on a projective variety $X$ satisfying \eqref{ss=s(U)}. Although in our case we are of course not assuming that \eqref{ss=s(U)} is satisfied (since this is the reason we are doing blow-ups), we will only use results which remain valid even without the assumption \eqref{ss=s(U)}.

Let $U' = \operatorname{Stab}_U (x_0)$. Since $U$ is abelian, we can choose a complementary subgroup $U'^{\perp}$ of $U'$ in $U$, so that $U = U' \times U'^{\perp}$. Then, for $x$ in a small enough neighbourhood $N_0$ of $x_0$ in $\xmino$ (which we can take to be invariant under $U'^{\perp}$), we have that $U'^{\perp}$ is complementary to $\operatorname{Stab}_U(x)$ if and only if $x \in C_{\operatorname{max}}(\xmino,U)$. Moreover the action of $U'^{\perp}$  on $N_0$ has trivial stabilisers, so that $N_{0}$ is contained in the stable locus for the action of $U'^{\perp}$ on $\xmino$ (by \cite[Thm 8.16 (1)]{Berczi2020}). As a result, by \cite[Prop 7.1]{Berczi2020} we obtain a geometric quotient $\pi: N_0 \to Y := N_0 / U'^{\perp}$, which has an induced linear action of $\hU / U'^{\perp} \cong U / U'^{\perp} \rtimes \GG_m$.

We now show that given $x \in N_0$ and $y = \pi(x) \in Y$, we have that $x \in  C_{\operatorname{max}}(\xmino,U) \cap U \zmin$ if and only if $y \in Y^{\hU / U'^{\perp} }$. If $x \in N_0$, then by the construction of $N_0$ we can assume that $U'^{\perp}$ is complementary to $\operatorname{Stab}_U (x)$ so that $U = U'^{\perp} \times \operatorname{Stab}_U(x)$. Note that there is an equality \begin{equation} \operatorname{Stab}_{U/U'^{\perp}}(\pi(x)) = (U'^{\perp} \times \operatorname{Stab}_U(x) ) / U'^{\perp}. \label{staby}
\end{equation} Thus if $x \in C_{\operatorname{max}}(\xmino,U)$ then $y = \pi(x)$ must be fixed by $(U'^{\perp} \times \operatorname{Stab}_U(x))/U'^{\perp} = U /U'^{\perp}$. If moreover $x \in U \zmin$, then $y = \pi(x)$ is fixed by the whole group $\hU / U'^{\perp}$, since $x$ is also fixed by a $\GG_m$ inside $\hU$. To show the converse, suppose that $x \in N_0$ does not have maximal dimension stabiliser group in $U$. Using \eqref{staby} again we see that $y$ cannot be fixed by all of $U/ U'^{\perp}$, since in this case $U'^{\perp}  \times \operatorname{Stab}_U(x) \subseteq U$. Therefore $y$ is not fixed by $\hU / U'^{\perp} $. This shows that $x \in C_{\operatorname{max}}(\xmino,U) \cap U \zmin$ if and only if $y \in Y^{\hU / U'^{\perp} }$.

Since $\pi$ is a geometric quotient with trivial stabiliser groups, it follows that $C_{\operatorname{max}}(\xmino,U) \cap U \zmin$ is smooth at $x_0$ if and only if $Y^{\hU / U'^{\perp} }$ is smooth at $y_0$. We note that since by assumption $X$ is smooth at $x_0$, we also have that $Y$ is smooth at $y_0$. 

By \cite[pp 33-34]{Berczi2020}, there exists an integer $s >0$ such that $Y$ admits a $\hU / U'^{\perp} $-equivariant locally closed $\GG_m$-equivariant embedding into the projective space $\mathbb{P}(W^{\vee})$ where $W = H^0(X,L^{\otimes s})^U$ (here $L$ denotes the ample line bundle on $X$ with respect to which the action is linearised). 
By \cite[Lem 7.6]{Berczi2020}, the image of the locally closed $\GG_m$-equivariant embedding $Y \to \mathbb{P}(W^{\vee})$ is contained in $\mathbb{P}(W^{\vee})^0_{\operatorname{min}}$, and this map restricts to a closed embedding $Y \to \mathbb{P}(W^{\vee})^0_{\operatorname{min}}$. We can therefore identify $Y$ as a subvariety of $\mathbb{P}(W^{\vee})^0_{\operatorname{min}}$.  We let $\overline{Y}$ denote the closure of $Y$ in this projective space. Then $\overline{Y}$ is a projective variety with a linear action of the externally graded unipotent group $\hU / U'^{\perp} $.

Since $y_0 \in Y$, we have that $y_0 \in \overline{Y}^0_{\operatorname{min}}  = \overline{Y} \cap \mathbb{P}(W^{\vee})^0_{\operatorname{min}}$. Moreover, from above we have that $y_0 \in Y^{\hU / U'^{\perp}}$ since $x_0 \in C_{\operatorname{max}}(\xmino, U) \cap U \zmin$. Therefore $y_0 \in  \overline{Y}^{\hU / U'^{\perp} } \cap \overline{Y}^0_{\operatorname{min}}  = \overline{Y}^{U/ U'^{\perp}} \cap Z(\overline{Y})_{\operatorname{min}}$. Since $Y$ is smooth at $y_0$, the projective variety $\overline{Y}$ is also smooth at $y_0$. We can therefore apply \thref{proofinspecialcase} to conclude that $\overline{Y}^{U/ U'^{\perp}} \cap Z(\overline{Y})_{\operatorname{min}}  =  \overline{Y}^{\hU / U'^{\perp} } \cap \overline{Y}^0_{\operatorname{min}}$ is smooth at $y_0$. And since $Y^{\hU / U'^{\perp} } \cap \overline{Y}^0_{\operatorname{min}}$ is an open subset of $\overline{Y}^{\hU / U'^{\perp} }$, it follows that $Y^{\hU / U'^{\perp} }$ is smooth at $y_0$. This implies by our argument that $C_{\operatorname{max}}(\xmino,\hU)$ is smooth at $x_0$. 
\end{proof}  

\subsection{The externally graded unipotent case} 

\label{subsec:simplesnonred}

If an externally graded unipotent group $\hU := U \rtimes \GG_m$ acts linearly on an irreducible projective variety $X$ and \eqref{ss=s(U)} is not satisfied, then provided \eqref{snonempty} is satisfied, by \thref{uhatwbups} a sequence of equivariant blow-ups of $X$ can be performed to obtain a variety $\hX$ for which \eqref{ss=s(U)} is satisfied for an induced linear action of $\hU$ on $\hX$, so that $\hX \gitq \hU$ is well-defined. In this section we establish a formula (see \thref{nonredformula}) for the Poincar\'e series of $\hX \gitq \hU$ in terms of cohomological information about $X$, under the assumptions that $X$ is smooth, that $U$ is abelian and that $\GG_m$ acts with a single weight on $\operatorname{Lie} U$, proving parts \ref{part1} and \ref{part2} of \thref{maintheorem} in the process.

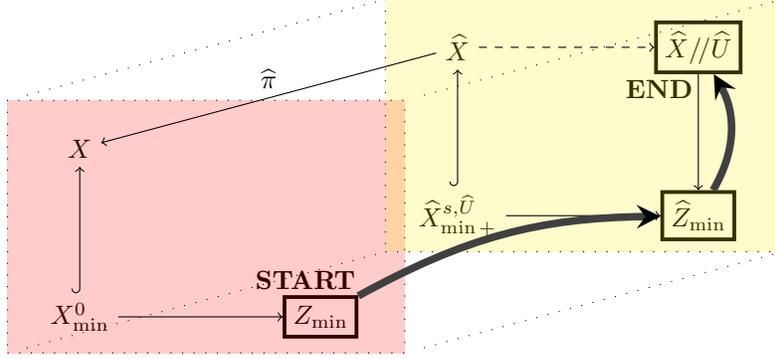
\begin{figure}
    \centering

		\begin{tikzpicture}[scale=0.8,x=(15:6.5cm),y=(0:4cm), z=(90:4cm),
    cross line/.style={preaction={draw=white,-,line width=3pt}},
    equal/.style={double distance=1pt},
    roundnode/.style={circle, draw=green!50, fill=gray!5, line width = 0.7mm},
    calcarrow/.style={line width=1mm,scale=3,>=stealth,draw=green!50}]
    

\node (c1) at (1,-1.3,0.55) {  } ;
\node (c2) at (1,-1.3,1.6) { };
\node (c3) at (1,0.35,0.55) { };
\node (c4) at (1,0.35,1.6) { };
\node (c5) at (0,-1.3,0.55) { } ;
\node (c6) at (0,-1.3,1.6) { } ;
\node (c7) at (0,0.35,0.55) { } ;
\node (c8) at (0,0.35,1.6) { } ;

\node (1) at (0,-1,1.4) {\( X \)};
\node[draw,line width = 0.5mm,label={[black!100,shift={(0,-0.05,-0.01)},font=\bfseries]{START}}] (1') at (0,0,0.7) {\( \zmin \)};
\node (2) at (0,-1,0.7) {\( \xmino \)};
\node (16)[draw, line width = 0.5mm,label={[black!100,shift={(0,-0.13,-0.29)},font=\bfseries]{END}}] at (1,0,1.4) {\(   \hX \gitq \hU  \)};
\node (19) at (1,-1,1.4) {\( \hX \)};
\node[draw,line width=0.5mm] (19') at (1,0,0.7) {\( \widehat{Z}_{\operatorname{min}}  \)};
\node (20) at (1,-1,0.7) {\(  \hX^{s,\hU}_{\operatorname{min}+} \)};

\filldraw  [fill = red, draw=black,loosely dotted, opacity=0.2] (c5) rectangle (c8);
\filldraw  [fill = yellow, draw=black,loosely dotted, opacity=0.2] (c1) rectangle (c4);
  
\draw [black,loosely dotted] (c1) rectangle (c4);
\draw  [black,loosely dotted] (c5) rectangle (c8);
\draw [black, loosely dotted] (c3) -- (c7);
\draw [black, loosely dotted] (c4) -- (c8);
\draw [black, loosely dotted] (c1) -- (c5); 
\draw [black, loosely dotted] (c2) -- (c6);

\draw [->] (20) -- (19');
\draw [right hook ->](2) -- (1);
\draw [->] (2) -- (1');

\draw [dashed, ->] (19) -- (16);

\draw [right hook ->] (20) -- (19);
\draw [->] (19) -- (1) node[pos=0.5,above] {$\widehat{\pi}$} ;

	\draw [calcarrow,->,draw=black!75] (1') to [bend left=14] (19');

  \draw [calcarrow,->,draw=black!75] (19') to [bend right=30] (16);
\draw[->] (16) to (19');

		\end{tikzpicture}

	\caption{Strategy for computing the Poincar\'e series of $\hX \gitq G$ in terms of the Poincar\'e series of $\zmin$, for an externally graded unipotent group $\hU = U \rtimes \GG_m$ acting linearly on a smooth complex projective variety $X$ such that there exists a point $z \in \zmin$ with $\operatorname{Stab}_U (z) = \{e\}$. The sequence of bold arrows indicates that the Poincar\'e series of $\hX \gitq \hU$ can be computed from the Poincar\'e series of $\widehat{Z}_{\operatorname{min}}$ using \eqref{uhatformula}, which can in turn be computed from the Poincar\'e series of $\zmin$ by studying the blow-up construction of \thref{uhatwbups}. 	}
 \label{fig:ssneqsuhat} 
		
\end{figure}

To prove \thref{nonredformula}  
we will proceed in two steps: first by establishing in Section \ref{afteronebu} a formula for the Poincar\'e series after a single blow-up (see \thref{onebu}), and then by applying \thref{onebu} iteratively for each stage of the blow-up in Section \ref{combining}. 

\subsubsection{Poincar\'e series after one blow-up}  \label{afteronebu}

The statement of \thref{onebu} relies on the following notation, which we will also use in \thref{nonredformula}.  

\begin{notation} \thlabel{notation1}
In \thref{cmaxnotation}, given the action of a group $H$ on a variety $Y$, we defined for any $d \in \mathbb{N}$ the subvariety $C_{d}(Y,H)$ of $Y$ consisting of points in $Y$ with $d$-dimensional stabiliser group in $H$. We now also define for any $d \in \mathbb{N}$ the subvariety $$C_{\geq d}(Y,H) : = \left\{ y \in Y \ | \ \operatorname{dim} \operatorname{Stab}_{H} (y) \geq d \right\} $$ of $Y$, and we note that this is a closed subvariety of $Y$ by the upper semi-continuity of stabiliser dimensions. 

When $H = \hU$ is an externally graded unipotent radical acting linearly on a projective variety $X$, then there is an equality   \begin{equation} C_{\geq d}(\xmino,\hU) \cap \zmin = C_{\geq d-1}(\zmin,U). \label{simplifyingeq} \end{equation}  Indeed, the grading condition on $\hU$ implies that $$\operatorname{dim} \operatorname{Stab}_U (p(x)) \geq \operatorname{dim} \operatorname{Stab}_U (x)$$ for any $x \in \xmino$, where we recall that $p: \xmino \to \zmin$ is the retraction map determined by the grading one-paramater subgroup of $\hU$.  

\end{notation} 

In \thref{onebu} below we give a formula for the Poincar\'e series of the $\zmin$ obtained after a single blow-up. The set-up is as follows. Let $\hU = U \rtimes \GG_m$ where the unipotent radical $U$ is abelian and $\GG_m$ acts with strictly positive weights on $\operatorname{Lie} U$ via the adjoint action. Suppose that $\hU$ acts linearly on a smooth projective variety $X$ and that \eqref{snonempty} is satisfied. Let $\overline{Y}$ denote the blow-up of $X$ along the closure of $C_{\operatorname{max}}(\xmino,\hU)$ in $X$, with a linear $\hU$-action obtained by pulling back the linearisation of the $\hU$-action on $X$ and perturbing by a small multiple $\epsilon$ of the exceptional divisor. Let $\pi: Y  \to \xmino$ denote the restriction of the blow-up to $\xmino$, or equivalently the blow-up of $\xmino$ along $C_{\operatorname{max}}(\xmino,\hU)$. By choosing $\epsilon$ small enough, the open subvariety $\overline{Y}^0_{\operatorname{min}}$ of $\overline{Y}$ will be contained in $Y$ (see \cite[Prop 8.5]{Berczi2020}). Since it is the open subvariety $\overline{Y}^0_{\operatorname{min}}$ that we are interested in, rather than the whole of $\overline{Y}$, to simplify notation we work with $Y$ instead of $\overline{Y}$. Nevertheless, where the definitions require a projective variety, the variety $Y$ should be replaced with $\overline{Y}$.

\begin{prop}[Poincar\'e series after one blow-up] \thlabel{onebu} 
Suppose that an externally graded abelian unipotent group $\hU$, where $U$ is abelian and $\GG_m$ acts with a single weight on $\operatorname{Lie} U$, acts linearly on a smooth complex projective variety $X$ such that \eqref{snonempty} is satisfied, and let $\pi: Y \to \xmino$ denote the blow-up of $\xmino$ along $C_{\operatorname{max}}(\xmino,\hU)$. 
Then we have: \begin{enumerate}[(i)]
\item the centre $C_{\operatorname{max}}(\xmino,\hU)$ of the blow-up is smooth, as is its intersection $C_{\operatorname{max}}(\zmin,U)$ with $\zmin$;
\item the Poincar\'e series of $Z(Y)_{\operatorname{min}}$ is given by $$P_t(Z(Y)_{\operatorname{min}}) = P_t(\zmin) + \frac{t^2 ( 1 - t^{2 d^0})}{1-t^2} P_t (C_{\operatorname{max}}(\zmin,U))$$ where $d^0 = \operatorname{codim}(C_{\operatorname{max}}(\zmin,U),\zmin)$;
\item the closed subvariety $C_{\operatorname{max}}(Z(Y)_{\operatorname{min}},U)$ of $Z(Y)_{\operatorname{min}}$ is a resolution of singularities of the closed subvariety $C_{\geq d_1-1}(\zmin,U)$ of $\zmin$, where $d_1$ is the second largest dimension of stabiliser group for points in $\xmino$. \label{resofsing}
\end{enumerate}
\end{prop} 

\begin{proof} 
As seen in Section \ref{subsec:ssneqsnonred}, the first step of the blow-up construction of \thref{uhatwbups} is to blow $\xmino$ up along the closed subvariety $C_{\operatorname{max}}(\xmino,\hU)$, namely the locus of points in $\xmino$ with maximal dimension unipotent stabiliser group\footnote{We recall from the discussion preceding the statement of \thref{onebu} that we must in fact blow $X$ up along the closure of $C_{\operatorname{max}}(\xmino,\hU)$ to ensure that we obtain a projective variety $\overline{Y}$. Since $\overline{Y}^0_{\operatorname{min}}$ is contained in $Y$ provided the linearisation is chosen appropriately, for simplicity we work with $Y$ only.}. We let $\pi : Y \to \xmino$ denote the blow-up of $\xmino$ along $C_{\operatorname{max}}(\xmino,\hU)$, and let $E$ denote the exceptional divisor. Since $C_{\operatorname{max}}(\xmino,\hU)$ is $\hU$-invariant, there is an induced action of $\hU$ on $Y$. As in the reductive case, we consider the linearisation of this action given by pulling back along the blow-up map the linearisation of the $\hU$-action on $\xmino$ and taking a tensor product with a sufficiently small multiple $\epsilon$ of the exceptional divisor. By \cite[Prop 8.8]{Berczi2020},nthe subvariety $Z(Y)_{\operatorname{min}}$ is the proper transform of $\zmin$ under this blow-up (or equivalently the blow-up of $\zmin$ along $C_{\operatorname{max}}(\xmino,\hU) \cap \zmin = C_{\operatorname{max}}(\zmin,U)$). 

Since $X$ is smooth, the subvariety $\xmino$ is also smooth and therefore by \thref{smoothnessofbulocus} we have that $C_{\operatorname{max}}(\xmino,\hU)$ is smooth. Moreover, by considering the action of $\GG_m$ on the closed subvariety $C_{\operatorname{max}}(\xmino,\hU)$ of $\xmino$, by the results of \cite{Bialynicki-Birula1973} we can conclude that $C_{\operatorname{max}}(\xmino,\hU) \cap \zmin = C_{\operatorname{max}}(\zmin,U)$ is also smooth. It follows that both $Y$ and $Z(Y)_{\operatorname{min}}$ are smooth, since they are the blow-ups of smooth varieties along smooth subvarieties.  

Using the properties of Poincar\'e series under blow-ups for smooth subvarieties, we obtain the desired formula for the equivariant Poincar\'e series of $Z(Y)_{\operatorname{min}} = \operatorname{Bl}_{C_{\operatorname{max}}(\zmin,U)} \zmin$: \begin{align} 
P_t(Z(Y)_{\operatorname{min}}) &  = P_t(\zmin) + P_t(E|_{Z(Y)_{\operatorname{min}}}) - P_t (C_{\operatorname{max}}(\zmin,U)) \nonumber \\
& = P_t(\zmin) + P_t (C_{\operatorname{max}}(\zmin,U)) \left( t^2 + t^4 + \cdots  +  t^{2(\operatorname{codim}( C_{\operatorname{max}}(\zmin,U), \zmin)-1)} \right) \nonumber\\
& = P_t(\zmin) + \frac{t^2 ( 1 - t^{2 d^{0}})}{1-t^2} P_t (C_{\operatorname{max}}(\zmin,U)). \label{formula1}
\end{align} 

To prove \ref{resofsing}, namely that $C_{\operatorname{max}}(Z(Y)_{\operatorname{min}},U)$ is a resolution of singularities of $C_{\geq d_{1}-1}(\zmin,U)$, it suffices to show that $C_{\operatorname{max}}(Z(Y)_{\operatorname{min}},U)$ is a blow-up of $C_{\geq d_{1}-1}(\zmin,U)$. Indeed, since $Z(Y)_{\operatorname{min}}$ is smooth, by \thref{smoothnessofbulocus} it follows that $C_{\operatorname{max}}(Z(Y)_{\operatorname{min}},U)$ is also smooth. We will therefore show that $C_{\operatorname{max}}(Z(Y)_{\operatorname{min}},U)$ is the blow-up of $C_{\geq d_{1}-1}(\zmin,U)$ along $C_{\operatorname{max}}(\zmin,U)$. This is equivalent to showing that $C_{\operatorname{max}}(Z(Y)_{\operatorname{min}},U)$ is the proper transform of $C_{\geq d_{1}-1}(\zmin,U)$ with respect to the blow-up $\pi: Y \to \xmino$. And this result follows from two observations. First, that there is an equality $d_1= d_{\operatorname{max}}(Y^0_{\operatorname{min}},\hU)$. Second, that by \cite[Prop 8.8 (c)]{Berczi2020} there is a strict inequality $\operatorname{dim} \operatorname{Stab}_U(z) <  d_{\operatorname{max}}(\zmin,U)$ for every $z \in Z(Y)_{\operatorname{min}}$.
\end{proof}

\subsubsection{General formula}  \label{combining}
We now give a general formula for the Poincar\'e series of $\hX \gitq \hU$, by applying \thref{onebu} iteratively for each stage of the blow-up construction (see \thref{nonredformula}). \thref{nonredformula} is the precise formulation of parts \ref{part1} and \ref{part2} of \thref{maintheorem}. In addition to \thref{notation1}, we will use the following notation. 

\begin{notation}
Suppose that an externally graded abelian unipotent group $\hU$ acts linearly on a projective variety $X$. We let $$d_{\operatorname{min}}(\xmino,\hU) := d_{(r+1)} < d_{(r)} < \cdots < d_{(1)} < d_{(0)} : = d_{\operatorname{max}}(\xmino,\hU)$$ denote the integers $d \in \mathbb{N}$ arising as the dimensions of stabiliser groups of points in $\xmino$. Moreover, for each $i = 0, \hdots, r+1$ we let \begin{equation} d^{i} := \operatorname{codim}(C_{\geq d_{i}}(\xmino,\hU) \cap \zmin, \zmin) = \operatorname{codim}(C_{\geq d_{i}-1}(\zmin,U),\zmin). \label{didef}
\end{equation}   
\end{notation}

\begin{theorem}[Poincar\'e series of $\hX \gitq \hU$ --  \thref{maintheorem} \ref{part1} \& \ref{part2}] \thlabel{nonredformula} 
Let $\hU = U \rtimes \GG_m$ where the unipotent radical $U$ is abelian and $\GG_m$ acts with strictly positive weights on $\operatorname{Lie} U$ via the adjoint action. Suppose that $\hU$ acts linearly on a smooth complex projective variety $X$ such that \eqref{snonempty} is satisfied.  
Let $$\hX := X_{(r+1)} \to X_{(r)} \to \cdots \to X_{(0)} = X$$ denote the variety resulting from applying the sequence of blow-ups of \thref{uhatwbups} to the action of $\hU$ on $X$ (for simplicity we let $\zmin^{i}$ denote the analogue of $\zmin $ for each $X_{i}$, namely $\zmin^{i} := Z(X_{i})_{\operatorname{min}}$). Then we have: \begin{enumerate}[(i)]
\item at each stage $i$ of the blow-up construction, the centre $C_{\operatorname{max}}((X_{i})^0_{\operatorname{min}},\hU)$ of the blow-up is smooth, as is its intersection $C_{\operatorname{max}}(\zmin^{i},U)$ with $\zmin^{i}$ (see \eqref{simplifyingeq});
\item the Poincar\'e series of $\hX \gitq \hU$ is given by
\begin{equation*} 
P_t(\hX \gitq \hU ) = \frac{1- t^{2d}}{1-t^2} \left(  P_t(\zmin) + \sum_{i=0}^{r}   t^2(1 -  t^{d^{i}})  P_t  \left( C_{\operatorname{max}}(\zmin^{i},U)  \right) \right)
\end{equation*} where $d : = \operatorname{dim} X - \operatorname{dim} U - \operatorname{dim} \zmin$ and $d^{i} : = \operatorname{codim}( C_{\geq d_{i}-1}(\zmin,U),\zmin)$;
\item \label{part3'} the closed subvariety $C_{\operatorname{max}}(\zmin^{i},U) $ of $\zmin^{i}$  is a resolution of singularities of the closed subvariety $C_{\geq d_{i}-1} (\zmin,U)$ of $\zmin$ for each $i =1 , \hdots r+1$.
\end{enumerate} 
\end{theorem}

As depicted in Figure \ref{fig:ssneqsuhat}, given the linear action of an externally graded unipotent group $\hU$ on a smooth projective variety $X$ satisfying the conditions needed so that \thref{uhatwbups} applies, the Poincar\'e series of the blown-up quotient $\hX \gitq \hU$ can be computed from that of $\zmin$: the aim is to determine from $P_t(\zmin)$ the Poincar\'e series of $\widehat{Z}_{\operatorname{min}}$ (which determines the Poincar\'e series of $\hX \gitq \hU$) by tracking through each step of the blow-up construction.

\begin{proof}
The first step of the blow-up construction is to blow $\xmino$ up along $C_{\operatorname{max}}(\xmino,\hU)$, giving $\pi_{1}: X_{1} \to \xmino$, and by \thref{onebu} we have a formula for the Poincar\'e series $P_t(\zmin^{1})$. 
Proceeding iteratively for each $i = 1,\hdots, r+1$, the variety $X_{i+1}$ is obtained by blowing $X_{i}$ along the smooth subvariety $C_{\operatorname{max}}((X_{i})^0_{\operatorname{min}},U)$. Note that the maximal dimension of unipotent stabiliser groups for points in $(X_{i})^0_{\operatorname{min}}$ is given by $d_{i}$. 

Applying the formula of \ref{resofsing} to $X_{i+1}$ instead of $X$ yields  \begin{equation}
P_t(\zmin^{i+1}) = P_t(\zmin^{i}) + \frac{t^2 ( 1 - t^{2 \operatorname{codim}(C_{\operatorname{max}}(\zmin^{i},U),\zmin^{i})}}{1-t^2} P_t \left( C_{\operatorname{max}}(\zmin^{i},U) \right). \label{intermediateform}
\end{equation}

By \thref{onebu} \ref{resofsing} applied at the $i$-th stage, we have that $C_{\operatorname{max}}(\zmin^{i},U)$ is a resolution of singularities of the closed subvariety $C_{\geq d_{i}-1}(\zmin^{i-1},U)$ of $\zmin^{i-1}$. Now $C_{\geq d_{i}-1}(\zmin^{i-1},U)$ can itself be identifed as the proper transform of $C_{\geq d_{i}-1}(\zmin,U)$ along the sequence of blow-ups $X_{i} \to X_{i-1} \to \cdots X_{(0)} = X$. Thus $C_{\geq d_{i}-1}(\zmin^{i-1},U)$ can be viewed as a resolution of singularities of $C_{\geq d_{i}-1}(\zmin,U)$ and moreover we have an equality $$\operatorname{codim}(C_{\operatorname{max}}(\zmin^{i},U),\zmin^{i-1}) = d^{i} : = \operatorname{codim}(C_{\geq d_{i}-1}(\zmin,U),\zmin).$$

The variety $X_{r+1}$ obtained at the $r+1$-th stage satisfies the property that points in $(X_{r+1})^0_{\operatorname{min}}$ have unipotent stabiliser groups of constant dimension equal to $d_{\operatorname{max}}((X_{r+1})^0_{\operatorname{min}},\hU) = d_{r+1}$, and so $$C_{\operatorname{max}}((X_{r+1})^0_{\operatorname{min}},\hU) = (X_{r+1})^0_{\operatorname{min}}.$$ Thus $X_{r+1}$ is the desired smooth variety $\hX$ admitting a linear $\hU$-action which satisfies the condition \eqref{ss=s(U)}\footnote{To be exact, the variety $\hX$ is in fact $\overline{X_{r+1}}$, see the discussion preceding the statement of \thref{onebu}.}.

We can therefore apply the formula \eqref{uhatformula} introduced in Section \ref{subsec:cohomuhat} to the action of $\hU$ on $\hX$, valid when \eqref{ss=s(U)} is satisfied:  \begin{equation*} P_t(\hX \gitq \hU) = P_t(\widehat{Z}_{\operatorname{min}}) \frac{1 - t^{2 \widehat{d}} }{1-t^2} 
\end{equation*}
where $\widehat{d} = \operatorname{codim}(U \widehat{Z}_{\operatorname{min}}, \hX) = \operatorname{dim} \hX - \operatorname{dim} U - \operatorname{dim} \widehat{Z}_{\operatorname{min}}$.  Since $\widehat{Z}_{\operatorname{min}}$ is the proper transform of $\zmin$ for the composition of blow-ups $\hX \to X$, we have that $\widehat{d}  = \operatorname{dim} X - \operatorname{dim} U - \operatorname{dim} \zmin$ so that $\widehat{d} = d$. Thus $$P_t(\hX \gitq \hU) = P_t(\widehat{Z}_{\operatorname{min}}) \frac{1 - t^{2 d} }{1-t^2}.$$ Finally, the Poincar\'e series $P_t(\widehat{Z}_{\operatorname{min}})$ can be computed by using \eqref{intermediateform} at each stage of the blow-up construction, thus giving the desired formula for $P_t(\hX \gitq \hU)$: \begin{equation*}P_t(\hX \gitq \hU ) = \frac{1- t^{2d}}{1-t^2} \left(  P_t(\zmin) + \sum_{i=0}^{r}t^2 ( 1 - t^{2 d^{i}})  P_t \left( C_{\operatorname{max}}(\zmin^{i},U \right) \right). 
\end{equation*} 
\end{proof}

\subsection{The general non-reductive case} \label{subsec:gennonred}

In this section we show how the result of Section \ref{subsec:simplesnonred} above can be combined with the results of Section \ref{subsec:cohomH}  to obtain a formula for the Poincar\'e series $\hX \gitq H$ under the assumption that semistability coincides with stability for the action of $R_{\l} : = R/ \lambda(\GG_m)$ on $\zmin$. That is, we prove the following \thref{mostgeneralformula}, which is the precise formulation of part \ref{part3} of \thref{maintheorem}. 

\begin{theorem}[Poincar\'e series of $\hX \gitq H$ -- \thref{maintheorem} \ref{part3}] \thlabel{mostgeneralformula}
Let $H =  U \rtimes R$ be an internally graded linear algebraic group where $U$ is abelian and the grading one-parameter subgroup $\lambda: \GG_m \to Z(R)$ acts with a single weight on $\operatorname{Lie} U$. Suppose that $H$ acts linearly on a smooth complex projective variety $X$ such that \eqref{snonempty} is satisfied and such that the semistable locus coincides with the stable locus (and is non-empty) for the induced action of $R_{\l} : = R/\lambda(\GG_m)$ on $\zmin$. 
Let $\hX$ denote the result of applying the blow-up construction of \thref{uhatwbups} to the action of $\hU: = U \rtimes \lambda(\GG_m) \subseteq H$ on $X$. Then the quotient $\hX \gitq H : = (\hX \gitq \hU ) \gitq R_{\l}$ has at worse finite quotient singularities and, using the notation from \thref{nonredformula}, its Poincar\'e series is given by: \begin{equation} P_t(\hX \gitq H )   = \frac{1- t^{2d}}{1-t^2}  \left(   P_t \left( \zmin \gitq R_{\l} \right) + \sum_{i=0}^{r}t^2 ( 1 - t^{2 d^{i}})  P_t \left( C_{\operatorname{max}}(\zmin^{i},U ) \gitq R_{\l}  \right) \right), \label{lastformula}
\end{equation}
 where $d : = \operatorname{dim} X - \operatorname{dim} U - \operatorname{dim} \zmin$ and $ d^{i} : = \operatorname{codim} ( C_{\operatorname{max}}(\zmin^{i},U ) \gitq R_{\l} ,\zmin \gitq R_{\l}).$ 
\end{theorem}  

 \thref{mostgeneralformula} shows that the Poincar\'e series of $\hX \gitq H$ can be computed in terms of information about $\zmin \gitq R_{\l}$, namely its Poincar\'e series and those of a finite number of blow-ups of $\zmin \gitq R_{\l}$. Figure \ref{fig:ssneqsgeneralsimp} illustrates this strategy for computing $P_t( \hX \gitq H)$.

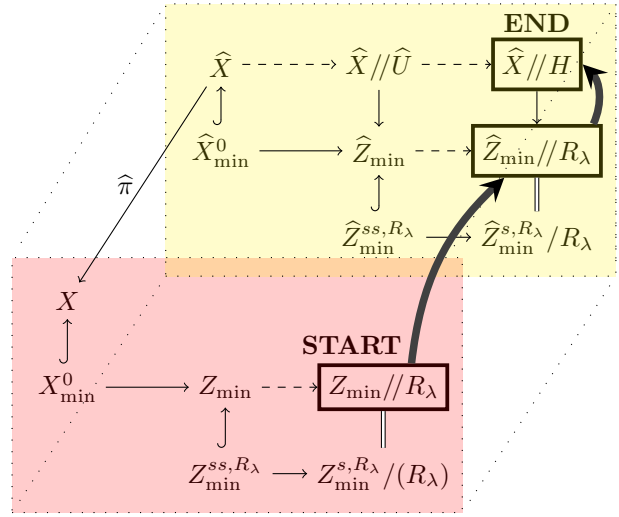
\begin{figure}[h]
\centering
 \begin{tikzpicture}[scale=0.5,x=(57:7.5cm),y=(0:4.2cm), z=(90:2.3cm),
    cross line/.style={preaction={draw=white,-,line width=3pt}},
    equal/.style={double distance=1pt},roundnode/.style={circle, draw=green!50, fill=gray!5, line width = 0.7mm},
    calcarrow/.style={line width=1mm,scale=3,>=stealth,draw=green!50}, dot/.style = {ellipse, draw=green!50,line width = 0.7mm, minimum width=45pt,minimum height=20pt,
              inner sep=0pt, outer sep=0pt},
dot/.default = 30pt  
, dot3/.style = {ellipse, draw=green!50,line width = 0.7mm, minimum width=30pt,minimum height=20pt,
              inner sep=0pt, outer sep=0pt},
dot3/.default = 30pt , dot/.style = {ellipse, draw=green!50,line width = 0.7mm, minimum width=45pt,minimum height=20pt,
              inner sep=0pt, outer sep=0pt},
dot/.default = 30pt 
                    ]

\node (c2) at (-1,1.4,2.5) { };
\node (c3) at (-1,-1.35,2.5) { };
\node (c7) at (-1,1.5,-0.45) { } ;
\node (c8) at (-1,-1.35,-0.45) { } ;
\node (c9) at (0,-1.35,2.7) {};
\node (c11) at (0,1.5,2.7) {};
\node (c10) at (0,-1.35,-0.45) {};
\node (c12) at (0,1.5,-0.45) { };
      
\node (1) at (-1,-1,2) {\( X \)};
\node (2') at (-1,-1,1) {\(\xmino \)};
\node  (3') at (-1,0,1) {\( \zmin \)};
\node (30) at (-1,0,0) {\( \zmin^{ss, R_{\l}} \)};
\node (30') at (-1,1,0) {\( \zmin^{s,R_{\l}}/(R_{\l}) \)};
\node (8) at (0,-1,2) {\( \hX \)};

\node[draw,line width=0.5mm,label={[black,shift={(0,-0.1,0)},font=\bfseries]{START}}] (8') at (-1,1,1) {\( \zmin \gitq R_{\l} \)};

\node (9) at (0,0,2) {\( \hX \gitq \hU \)};
\node[draw,line width=0.5mm,label={[black,shift={(0,00,-0.02)},font=\bfseries]{END}}] (10) at (0,1,2) {\(  \hX \gitq H \)};

\node (11) at (0,-1,1) {\( \hX^0_{\operatorname{min}} \)};

\node (12) at (0,0,1) {\( \widehat{Z}_{\operatorname{min}} \)};
\node (13) at (0,0,0) {\( \widehat{Z}_{\operatorname{min}}^{ss,R_{\l}} \)};
\node (14) at (0,1,0) {\( \widehat{Z}_{\operatorname{min}}^{s,R_{\l}} / R_{\l} \)};
\node[draw,line width=0.5mm] (13') at (0,1,1) {\( \widehat{Z}_{\operatorname{min}} \gitq R_{\l} \)};

\filldraw  [fill = red, draw=black,loosely dotted, opacity=0.2] (c3) rectangle (c7);
\filldraw  [fill = yellow, draw=black,loosely dotted, opacity=0.2] (c10) rectangle (c11);

\draw  [black,loosely dotted] (c10) rectangle (c11);
\draw  [black,loosely dotted] (c3) rectangle (c7);

\draw [black,loosely dotted] (c10) -- (c8);
\draw [black,loosely dotted] (c9) -- (c3);
\draw [black,loosely dotted] (c11) -- (c2);
\draw [black,loosely dotted] (c12) -- (c7);

\draw [dashed,->] (3') -- (8');
\draw [->] (30) -- (30');
\draw [equal] (30') -- (8');
\draw[dashed,->] (12) -- (13');
\draw[->] (13) -- (14);
\draw[equal] (14) -- (13');
\draw [right hook ->] (30) -- (3');

\draw [loosely dotted, black] (c11) -- (c12);
\draw [-> ] (2') -- (3');
\draw [right hook ->] (2') -- (1);
\draw [right hook ->](11) -- (8);

\draw [dashed,->] (8) -- (9);
\draw [dashed,->] (9) -- (10);
\draw [->] (9) -- (12);
\draw [->] (10) -- (13');

\draw [->] (8) -- (1) node[midway,left] {$\widehat{\pi}$} ;

\draw[->](11) -- (12);

\draw [right hook-> ](13) -- (12);

  \draw [calcarrow,->,draw=black!75] ([xshift=2.5mm]8'.north)  to [bend left=20] (13');
   \draw [calcarrow,->,draw=black!75] ([xshift=5mm]13'.north) to [bend right=35] (10.east);
 
		\end{tikzpicture} \caption{Strategy for computing the Poincar\'e series of $\hX \gitq H$ in terms of the Poincar\'e series of $\zmin \gitq R_{\l}$ when $\zmin^{ss,R_{\l}} = \zmin^{s,R_{\l}} \neq \emptyset$ and \eqref{snonempty} is satisfied. Here $H = U \rtimes R$ is an internally graded linear algebraic group with abelian unipotent radical $U$ and such that the grading one-parameter subgroup $\lambda: \GG_m \to Z(R)$ acts with a single weight on $\operatorname{Lie} U$ via the adjoint action, and $H$ is assumed to act linearly on a smooth complex projective variety $X$. The sequence of bold arrows indicate that the Poincar\'e series of $\hX \gitq H$ can be computed from the Poincar\'e series of $\zmin \gitq R_{\l}$, using the fact that $\widehat{Z}_{\operatorname{min}} \gitq R_{\l}$ can be obtained from a sequence of blow-ups of $\zmin \gitq R_{\l}$, and that the Poincar\'e series of $\widehat{Z}_{\operatorname{min}} \gitq R_{\l}$ determines that of $\hX \gitq H$ by \eqref{Hformula}. 
		}
		
		 \label{fig:ssneqsgeneralsimp}

	 	\end{figure}

	 \begin{rk}[Methods for computing $P_t(\zmin \gitq R_{\l})$] \thlabel{computingquotient}
	As noted above, by \thref{mostgeneralformula} the computation of $P_t(\hX \gitq H)$ reduces to that of $P_t(\zmin \gitq R_{\l})$, which can be calculated from the $H$-equivariant Poincar\'e series of $X$ by using the results from Section \ref{subsec:cohomred}. Indeed the results from classical GIT summarised in Section \ref{subsec:cohomred} provide an inductive formula for computing the Poincar\'e series of a GIT quotient in terms of the equivariant Poincar\'e series of the parameter space -- see \eqref{redformula}. Applying the result in this case yields an inductive formula for $P_t(\zmin \gitq R_{\l})$ in terms of $P_t^{R_{\l}}(\zmin)$, which can itself be computed from the $H$-equivariant Poincar\'e series of $X$ using the retraction map $p:  \xmino \to \zmin$ and the fact that $\xmino$ is the open Bialynicki-Birula stratum for the action of $\lambda(\GG_m)$ on $X$. 
	
	However, when applying \thref{mostgeneralformula} to moduli problems it is not always necessary to use the strategy described above to compute $P_t(\zmin \gitq R_{\l})$. Indeed, the variety $\zmin \gitq R_{\l}$ typically corresponds to a moduli space for a subset of the objects to be classified, whose cohomology may already be known. This is the case for moduli spaces of unstable vector or Higgs bundles on a smooth projective curve (and more generally of sheaves or Higgs sheaves on a smooth projective variety), which can be constructed using Non-Reductive GIT. In the Non-Reductive GIT set-up for the construction of these moduli spaces, $\zmin \gitq R_{\l}$ is a moduli space for those unstable bundles which are isomorphic to their Harder-Narasimhan graded  (see \cite{Jackson2018,Hamilton2019a}). In \thref{applicationtobundles} below we discuss the application of \thref{mostgeneralformula} to this example. 
	 \end{rk}

\begin{rk}[Application of \thref{mostgeneralformula} to unstable Higgs or vector bundles] \thlabel{applicationtobundles} 
\thref{mostgeneralformula} can be used to compute the Poincar\'e series of moduli space which can be constructed as non-reductive GIT quotients of smooth varieties by non-reductive group actions satisfying the conditions of \thref{mostgeneralformula}. Examples of such moduli spaces are moduli spaces for unstable vector or Higgs bundles on a smooth projective curve, and more generally (Higgs) sheaves on a smooth projective variety, with a Harder-Narasimhan type of length two (this includes the case of rank two unstable vector or Higgs bundles). Indeed such moduli spaces can be constructed using Non-Reductive GIT (see \cite{Jackson2018,Hamilton2019a}), a construction which requires performing the non-reductive blow-ups as the condition \eqref{ss=s(U)} is not satisfied. 

In these examples the quotient $\zmin \gitq R_{\l}$ can be interpreted as a product of moduli spaces of semistable Higgs/vector bundles  parametrising Higgs/vector bundles of a fixed Harder-Narasimhan type which are isomorphic to their Harder-Narasimhan graded, as noted already in \thref{computingquotient} above. Moreover, the centres of the iterated blow-ups of $\zmin \gitq R_{\l}$ are proper transforms of closed subvarieties of $\zmin \gitq R_{\l}$ which can be constructed from certain Brill-Noether loci associated to the base curve. We note that by part \ref{part3'} of \thref{nonredformula}, these proper transforms have at worst finite quotient singularities (since they are geometric quotient with finite stabiliser groups of smooth varieties), and thus represent partial desingularisations of the corresponding Brill-Noether loci. In the case of rank two bundles, the Brill-Noether loci appearing in the setting of \thref{mostgeneralformula} are of rank one; they are of higher rank when considering higher rank bundles. 
\end{rk}

We now prove \thref{mostgeneralformula}. 

\begin{proof}[Proof of \thref{mostgeneralformula}] 
Performing the sequence of non-reductive blow-ups of \thref{uhatwbups} results in a variety $\hX$ with a linear action of $H$ such that \eqref{ss=s(U)} is satisfied, so that $\hX \gitq \hU$ is well-defined. The non-reductive GIT quotient by $H$ is given by $\hX \gitq H : =  (\hX \gitq \hU) \gitq R_{\l}$. Our aim is to apply \eqref{Hformula} from Section \ref{subsec:cohomH}, which provides a formula for the Poincar\'e series of $\hX \gitq H$ in terms of $\widehat{Z}_{\operatorname{min}} \gitq R_{\l}$. To do so we must know that two conditions are satisfied: firstly that $\hX^0_{\operatorname{min}}$ is smooth\footnote{In Section \ref{subsec:cohomuhat} we assumed the stronger condition that $X$ is smooth. However, the results stated only require that $\xmino$ is smooth since the equivariantly perfect stratification used to compute the Poincar\'e series of the quotient is a stratification only of $\xmino$ rather than of all of $X$.} and secondly that $\widehat{Z}_{\operatorname{min}}^{ss,R_{\l}} = \widehat{Z}_{\operatorname{min}}^{s,R_{\l}} \neq \emptyset$. 

The first condition is satisfied by \thref{smoothnessofbulocus}, since we have assumed that $X$ (and therefore $\xmino$) is smooth. Indeed $\widehat{X}^0_{\operatorname{min}}$ is an open subset of a variety $\hX$ obtained after a finite number of blow-ups of the smooth variety $\xmino$, each with a smooth centre by \thref{smoothnessofbulocus}. 

We now prove that the second condition is also satisfied, using the assumption that $\zmin^{ss,R_{\l}} = \zmin^{s,R_{\l}} \neq \emptyset$. By construction $\widehat{Z}_{\operatorname{min}}$ coincides with the proper transform of $\zmin$ in $\hX$, and thus can be viewed as a blow-up of $\zmin$. This perspective enables us to apply the arguments of \cite{Reichstein1989} which studies the behaviour of stability under blow-ups. In particular, it is shown there that stability is preserved under blowing up, from which it follows that the stable locus in $\widehat{Z}_{\operatorname{min}}$ coincides with the proper transform of the stable locus in $Z_{\operatorname{min}}$. Since $\zmin^{ss,R_{\l}} = \zmin^{s,R_{\l}} \neq \emptyset$ by assumption, it follows that $\widehat{Z}_{\operatorname{min}}^{ss,R_{\l}} = \widehat{Z}_{\operatorname{min}}^{s,R_{\l}} \neq \emptyset$. 

Thus by \eqref{Hformula} of Section \ref{subsec:cohomH} we have that $$P_t(\hX \gitq H)  = P_t(\widehat{Z}_{\operatorname{min}} \gitq R_{\l}) \frac{1 - t^{2d}}{1-t^2}.$$ It remains only to express $P_t(\widehat{Z}_{\operatorname{min}} \gitq R_{\l})$ in terms of $\zmin \gitq R_{\l}$ and of its blow-ups, which we can do by using the commutative diagram 
\begin{center}
 \begin{tikzcd} 
\widehat{Z}_{\operatorname{min}}   \arrow[d] & \widehat{Z}_{\operatorname{min}}^{s,R_{\l}} \arrow[l,hook] \arrow[d]  \arrow[r] & \widehat{Z}_{\operatorname{min}} \gitq R_{\l} \arrow[d]  \\
 \zmin   & \zmin^{s,R_{\l}} \arrow[l,hook] \arrow[r] & \zmin \gitq R_{\l} 
\end{tikzcd} 
\end{center}
where the vertical maps are sequences of blow-down maps. The commutativity of the square on the left follows from the arguments of \cite{Reichstein1989} evoked in the paragraph above and the assumption that $\zmin^{ss,R_{\l}} = \zmin^{s,R_{\l}}$. The commutativity of the square on the right follows from the fact that $\widehat{Z}_{\operatorname{min}} \gitq R_{\l}$, which is a geometric quotient given that $\widehat{Z}^{ss,R_{\l}} = \widehat{Z}_{\operatorname{min}}^{s,R_{\l}}$, can equivalently be viewed as a sequence of blow-ups of $\zmin \gitq R_{\l}$ along the closed subvarieties given by the images of those in $\zmin \cap \zmin^{s,R_{\l}}$ along the quotient map $\zmin^{s,R_{\l}} \to \zmin \gitq R_{\l}$. That is, at each stage $i$ the corresponding quotient $\zmin^{i} \gitq R_{\l}$ is obtained by blowing up the quotient $\zmin^{i+1} \gitq R_{\l}$ along the closed subvariety $C_{\operatorname{max}}(\zmin^{i},U) \gitq R_{\l}$, and so its Poincar\'e series is given by $$ P_t( \zmin^{i+1} \gitq R_{\l}) = P_t(\zmin^{i} \gitq R_{\l}) + \frac{t^2 ( 1- t^{2d^{i-1}})}{1-t^2}  P_t(C_{\operatorname{max}}(\zmin^{i},U) \gitq R_{\l})$$ where $d^{i} : = \operatorname{codim} ( C_{\operatorname{max}}(\zmin^{i},U) \gitq R_{\l}, \zmin^{(i)} \gitq R_{\l}.$ Applying this formula iteratively for each stage in the blow-up construction, we obtain the desired formula \eqref{lastformula}. 
\end{proof}

We conclude this paper by providing a strategy for generalising \thref{mostgeneralformula} to the case where semistability does not coincide with stability for the action of $R_{\l}$ on $\zmin$. This situation occurs for example when applying Non-Reductive GIT to the construction of moduli spaces for unstable vector or Higgs bundles whose Harder-Narasimhan type is not coprime. 

\begin{rk}[Extending \thref{mostgeneralformula} to the case where semistability does not coincide with stability for the action of $R_{\l}$ on $\zmin$]
To establish the formula \eqref{lastformula} for computing the Poincar\'e series of $\hX \gitq H$ we have assumed that semistability coincides with stability for the action of $R_{\l}$ on $\zmin$, and that the stable locus is non-empty. This ensures that $\hX \gitq H$ has at worst finite quotient singularities if $X$ is smooth. Without this condition, the quotient $\hX \gitq H$ may have much worse singularities even if $X$ is smooth. In this case, provided we assume that $\zmin^{s, R_{\l}}$ is non-empty, we  can still compute its intersection Poincar\'e series\footnote{We consider intersection cohomology instead in this case as it is better suited to the study of singular projective varieties than ordinary cohomology.},  by using results from \cite{Kirwan1986a} for computing the intersection Poincar\'e series of classical GIT quotients when semistability does not coincide with stability. That is, we can follow-up the non-reductive blow-ups with the partial desingularisation construction of classical GIT for the action of $R$ on $\hX$, to obtain a variety $\tX$ such that semistability coincides with stability for the induced action of $H$ and such that $\widetilde{Z}_{\operatorname{min}}^{ss,R_{\l}} = \widetilde{Z}^{s,R_{\l}}_{\operatorname{min}}$. Moreover, if $X$ is smooth then $\tX$ will also be smooth, so that \eqref{Hformula} can be used to compute $P_t(\tX \gitq H)$ in terms of $\widetilde{Z}_{\operatorname{min}} \gitq R_{\l}$. The  quotient $\widetilde{Z}_{\operatorname{min}} \gitq R_{\l}$ can equivalently be viewed as the quotient obtained after applying the partial desingularisation construction to the action of $R_{\l}$ on $\widehat{Z}_{\operatorname{min}}$. Thus the formula from \cite{Kirwan1985} can be used to describe the Poincar\'e series of $\widetilde{Z}_{\operatorname{min}} \gitq R_{\l}$ in terms of the equivariant Poincar\'e series of $\widehat{Z}_{\operatorname{min}}$ (and of iterated blow-ups of it), which can in turn be computed from the equivariant Poincar\'e series of $\zmin$  (and of iterated blow-ups of it) by the proof of \thref{nonredformula}. 
\end{rk}

\bibliographystyle{abbrv}
\bibliography{References}

\end{document}